\documentclass[11pt]{amsart}
\usepackage{amsmath,amssymb,epsfig,color, amsthm, mathrsfs}
\usepackage{wrapfig,lipsum,floatrow,sidecap}
\usepackage{graphicx}
\usepackage{multirow}
\usepackage{hhline}
\usepackage{url}
\graphicspath{ {images/} }
\newcommand{\comment}[1]{}

\newtheorem{theorem}{Theorem}[section]

\newtheorem{definition}[theorem]{Definition}
\newtheorem{lemma}[theorem]{Lemma}
\newtheorem{proposition}[theorem]{Proposition}
\newtheorem{remark}[theorem]{Remark}
\newtheorem{example}[theorem]{Example}

\newcommand{\vanish}[1]{}\parskip=12pt

\newcommand{\sign}{\operatorname{sign}}
\newcommand{\signc}{\varepsilon}
\newcommand{\signt}{\operatorname{sign}_t}

\def\p{\prime}

\def\b{\textbf{b}}
\def\bb{\bar{b}}
\def\tb{\tilde{b}}

\def\H{\mathcal{H}}
\def\M{\mathcal{M}}

\numberwithin{equation}{section}

\begin{document}
\title{Invariants of rational links represented by reduced alternating diagrams}
\author{Yuanan Diao$^\dagger$, Claus Ernst$^*$ and Gabor Hetyei$^\dagger$}
\address{$^\dagger$ Department of Mathematics and Statistics\\
University of North Carolina Charlotte\\
Charlotte, NC 28223}
\address{$^*$ Department of Mathematics\\
Western Kentucky University\\
Bowling Green, KY 42101, USA}
\email{}
%\thanks{$^Dag$supported by NSF grant DMS-0712958.
%$^\ast$supported by NSA grant \# H98230-07-1-0073} Dedicatory{}

\begin{abstract}
A rational link may be represented by any of the (infinitely) many link diagrams corresponding to various
continued fraction expansions of the same rational number. The continued fraction expansion of the rational number in which all signs are the same is called a {\em nonalternating form} and the diagram
corresponding to it is a reduced alternating link diagram, which is minimum in terms of the number of crossings in the diagram. Famous formulas exist in the literature for the braid
index of a rational link by Murasugi and for its HOMFLY polynomial by
Lickorish and Millet, but these rely on a special continued fraction
expansion of the rational number in which all partial denominators are even (called {\em all-even form}). In 
this paper we present an algorithmic way to transform a continued
fraction given in nonalternating form into the all-even form. Using this
method we derive formulas for the braid index and the HOMFLY polynomial
of a rational link in terms of its reduced alternating form, or equivalently the nonalternating form of the corresponding rational number.    
\end{abstract}

\maketitle
\section{Introduction}

How to compute the braid index of rational links was first shown by Murasugi~\cite{Mu} more than 25 years ago.  This computation became possible because of the discovery of the HOMFLY polynomial $H(a,z)$~\cite{Fr,Pr}. Using this polynomial one could derive the Morton-William-Frank inequality~\cite{FW, Mo}: 
$$\textbf{b}(K)\ge (E(K)-e(K))/2+1,$$ 
where $K$ is a knot or link, $\b(K)$ is the braid index of $K$, $E(K)$ is the maximal $a$-power of $H_K(a,z)$, $e(K)$ is the minimal $a$-power of $H_K(a,z)$, and $H_K(a,z)$ is the HOMFLY polynomial of the knot or link $K$. In addition, we have an inequality given by Yamada~\cite{Ya}
$$\b(K)\le s(D_K),$$
where $D_K$ is any regular diagram of $K$ and $s(D_K)$ is the number of Seifert circles in $D_K$.
In the case that one can find a diagram $D$ of a knot or link $K$ such that
$s(D)= (E(D)-e(D))/2+1$ holds, it then follows that we must have $s(D)=\textbf{b}(D)$. 

When Murasugi~\cite{Mu} first established the braid index of rational
links, he did so by using a special diagram $D_K$ of the link $K$ that
relies on a particular continued fraction expansion using only even
integers which is usually highly non minimal. The same special expansion
was used by Lickorish and Millet~\cite[Proposition 14]{Li-Mi} to present
a formula for the HOMFLY polynomial of a rational link. A possible
reason for representing rational links with such a special diagram is noted by
Duzhin and Shkolnikov~\cite{Du} who observed that ``due to the fact that
all blocks are of even length, strands are everywhere
counter-directed.''  

Unfortunately, this special all-even expansion is not only highly non-minimal,
but it also does not exist if both integers in the fraction defining the
rational link are odd, in which case one could not apply this method
directly and  needs to use a diagram of the mirror image of $K$ (which
corresponds to a different fraction in which one of the integers is
even). Missing are formulas that are stated in terms of minimal (alternating) diagrams 
or equivalently that can use any fraction of a rational link (even those where both integers are odd).

A first effort in this direction was made by the authors in~\cite{DL},
which contains a formula for the braid index of a rational link $K$,
represented by a canonical minimal diagram $D$ of $K$. The proof
of the formula is quite different from Murasugi's approach, the formula
itself is in fact just one of the applications of a more
general result in~\cite{DL}, which can be applied to compute the braid
index of many other knot and link families, including all alternating
Montesinos links~\cite{DL}. The more general validity of the formula
also makes it directly applicable to a rational link $K$ (without using
its mirror image) even if the numerator and denominator of the
corresponding rational number are both odd. 

In this paper we take a completely different approach. We develop a
general procedure which transforms (when this is possible) a continued
fraction representing an alternating link diagram into a continued
fraction whose partial denominators are all even. A continued fraction
representing an alternating link diagram is a continued fraction in
which the signs of the partial denominators do not alternate. We call
these {\em nonalternating} continued fractions and we partition their
partial denominators into {\em primitive blocks}. The conversion into
all-even form may be performed on these primitive blocks essentially
independently (except for some easily predictable propagation of
signs). The details of these purely arithmetic manipulations are given
in Section~\ref{sec:conv}. With this conversion at hand we have a tool
that allows us to transform Murasugi's formula for the braid index~\cite{Mu}
and the Lickorish-Millet formula for the HOMFLY
polynomial~\cite{Li-Mi} directly. As explained in
Section~\ref{sec:transf}, the primitive blocks also have the property
that the crossing sings in an alternating rational link are constant
within a block and opposite in adjacent blocks. This observation allows
us to phrase our formulas in terms of the crossing signs in a manner
that is similar to the main result in~\cite{DL}. Our new braid index
formula is presented in Section~\ref{sec:braid} and a new HOMFLY
polynomial formula is presented in Section~\ref{HOMFLY_section}.   
The connection between our present braid index formula and the one
derived in~\cite{DL} is explained in Section~\ref{sec:connect}.

The proofs of the results presented in this paper rely on representing the link
diagram in all-even form even if they are stated in terms of a minimal
representation. Since taking the mirror image  does not change the braid
index and changes the HOMFLY polynomial only to the extent of a simple
substitution $a\mapsto a^{-1}$, the formulas we find remain useful even
in the case when the rational link diagram has no all-even
representation. It is an interesting question of future research to
develop a HOMFLY formula that is independent of the existence of the
all-even representation and is comparable in this sense to the braid
index formula presented in~\cite{DL}. 

\medskip

\section{Finite simple continued fractions}\label{sec:cf}

We define a %generalized
finite simple continued fraction as an
expression of the form 
\begin{equation}
  \label{eq:cf}
[c_0,\ldots,c_n]=c_0+\cfrac{1}{c_1+\cfrac{1}{c_2+\ddots\cfrac{1}{c_{n-1}+\cfrac{1}{c_n}}}}, 
\end{equation}
where the {\em partial denominators} $c_0,\ldots,c_n$ are
integers and $c_n\neq 0$. If $c_n\neq 0$, we can evaluate the 
expression using algebra and we obtain a rational number
$p/q$. Conversely, as it is well-known, every rational number $p/q$ may
be written as a finite continued fraction in such a way, that only the
first partial denominator $c_0$ may be zero or negative, all other partial
denominators are positive. This representation is unique up to the
possibility of replacing $[c_0,\ldots,c_{n-1},1]$ with
$[c_0,\ldots,c_{n-1}+1]$ or, conversely, replacing $[c_0,\ldots,c_n]$
where $c_n>1$ with $[c_0,\ldots,c_n-1,1]$.

Equivalently, we may introduce
$$
M(c)=\begin{pmatrix}
c & 1\\
1 & 0\\
\end{pmatrix}
$$
for each integer $c$, and then the following statement is easily shown
by induction on $n$ (cf.~\cite[p. 205]{Cromwell}).
\begin{proposition}
If the integers $c_0,c_1,\ldots, c_n$ satisfy $c_n\neq 0$,
then the value $p/q$ of $[c_0,\ldots,c_n]$ (for the integers $p$ and
$q$) is given by   
\begin{equation}
\begin{pmatrix}
p\\
q\\
\end{pmatrix}
=
M(c_0)M(c_1)\cdots M(c_n)
\begin{pmatrix}
1\\
0\\
\end{pmatrix}.
\label{eq:mdef}
\end{equation}  
\end{proposition}  
Using Equation~(\ref{eq:mdef}) we may extend the definition of the
evaluation $[c_0,c_1,\ldots,c_n]$ to any finite sequence
$c_0,c_1,\ldots,c_n$ of integers. We will set $p/0=\infty$ when $p\neq
0$ and we will leave $0/0$ undefined. Note that, the determinant of
$M(c)$ is $-1$, 
regardless of the value of $c$, and so the vector given in 
Equation~(\ref{eq:mdef}) is never the null vector.

It is easy to verify directly that extending the evaluation of
Equation~(\ref{eq:mdef}) to all finite sequences of integers amounts to
adding the following rules to the evaluation of~(\ref{eq:cf}). We set 
$$
\frac{p}{0}=\infty \quad\mbox{for $p\neq 0$ and}\quad \frac{p}{\infty}=0
\quad\mbox{all $p$}.
$$

A key equality whose variants we will be using is the following formula
of Lagrange (see Lagrange's Appendix to Euler's Algebra~\cite{Eu} cited
in~\cite{Ka-Go}),
\begin{equation}
\label{eq:Lagrange}  
[a,-b]=a-\frac{1}{b}=a-1-\frac{1}{1+\frac{1}{b-1}}=[a-1,1,b-1]. 
\end{equation}  

A slightly generalized variant of~(\ref{eq:Lagrange}) is the following.
\begin{proposition} 
  \label{prop:e}
For $\delta\in \{-1,1\}$, any generalized finite simple continued
fraction $[c_0,\ldots,c_n]$ satisfies
$$
[c_0,\ldots,c_i,\ldots,c_j,\ldots,
  c_n]
=[c_0,\ldots,c_i+\delta,-\delta,\delta-c_{i+1},
  -c_{i+2}, \ldots,-c_j,\ldots,-c_n].     
$$
\end{proposition}
This is a direct consequence of the equation
$$
M(c_{i})M(c_{i+1})
\begin{pmatrix}
  p\\
  q\\
 \end{pmatrix} 
=M(c_{i}+\delta)M(-\delta)M(\delta-c_{i+1})
\begin{pmatrix}
\delta  p\\
-\delta  q\\
 \end{pmatrix} 
$$
that holds for any pair of numbers $(p,q)$. The direct verification is
left to the reader.  

In our applications most of the time we will be interested in finite
simple continued fractions satisfying $c_1\cdots c_n\neq 0$.
\begin{definition}
We call a finite simple continued fraction $[c_0,c_1,\ldots,c_n]$ {\em
  nonsingular}  if it satisfies $c_1\cdots c_n\neq 0$, otherwise we call
it {\em singular}. 
\end{definition}  

Applying Proposition~\ref{prop:e} to a nonsingular
$[c_0,c_1,\ldots,c_n]$ results in a singular continued fraction only if
$c_i=-\delta$ or $c_{i+1}=\delta$. In all of these situations we
may return to nonsingular finite continued fractions by using the
following rule.
\begin{lemma}
\label{lem:add}
  We have
 $$
 [c_0,c_1,\ldots, c_{j-1},0,c_{j+1}, \ldots, c_n]
 =
  [c_0,c_1,\ldots, c_{j-1}+c_{j+1}, \ldots, c_n].
 $$
\end{lemma}  
This is a direct consequence of $M(c_{i-1})M(0)M(c_{i+1})=M(c_{i-1}+c_{i+1})$.
In the case when $c_i=-\delta$, Proposition~\ref{prop:e} yields
$$
[c_0,\ldots,c_{i-1},-\delta,c_{i+1}\ldots,c_j,\ldots,
  c_n]
=
[c_0,\ldots,c_{i-1},0,-\delta,\delta-c_{i+1},\ldots,
  -c_j,\ldots,-c_n].  
$$
Combining this equation with Lemma~\ref{lem:add} we obtain
$$
[c_0,\ldots,c_{i-1},-\delta,c_{i+1}\ldots,c_j,\ldots,
  c_n]
=
[c_0,\ldots,c_{i-1}-\delta,\delta-c_{i+1},\ldots,
  -c_j,\ldots,-c_n].  
$$
%This last equation is an instance of Proposition~\ref{prop:e} with the
%left and right hand sides swapped. 

Using Proposition~\ref{prop:e} and Lemma~\ref{lem:add} we may transform
any finite simple continued fractions in one of the two {\em standard
  forms} as defined in Definition~\ref{nonalternatingform} below. 

\begin{definition}
\label{nonalternatingform}
A finite simple continued fraction $[c_0,\ldots,c_n]$ is in {\em nonalternating
denominator form} if $c_1\cdots c_n\neq 0$, the integers $c_1,\ldots,
c_n$ all have the same sign, and $c_0$ is either zero or has the same
sign as all the other $c_i$-s. On the other hand, 
a finite continued fraction $[c_0,\ldots,c_n]$ is in {\em even
  denominator form} if $c_1\cdots c_n\neq 0$ and the integers
$c_0,\ldots,c_{n-1}$ are all even integers. 
\end{definition} 

The following statement is a variant of the well-known uniqueness result
on the standard continued fraction expansion of a rational number.

\begin{lemma}
\label{lem:nonaltu}  
Every rational number $p/q$ has a representation in the {\em
  nonalternating denominator form}. This form is unique up to the
possibility of replacing $[c_0,\ldots,c_n]$ with $[c_0,\ldots,c_n-1,1]$
when $c_n>1$ or with  $[c_0,\ldots,c_n+1,-1]$ when $c_n<-1$.
\end{lemma}
%Lemma~\ref{lem:nonaltu} coincides with the well-known uniqueness result
%for positive rational numbers. 

The nonalternating representations of a positive and a negative rational
number of the same absolute value are connected by 
the obvious equality 
$[-c_0,\ldots,-c_n]=(-1)\cdot [c_0,\ldots,c_n]$.    

The following statement is a generalization of the observation made in~\cite[Lemma 2]{Du}. Its proof is essentially the same as that given in~\cite[Lemma 2]{Du}. 

\begin{lemma}
\label{lem:evenden}  
Every rational number $p/q$ has a unique representation
$[c_0,\ldots,c_n]$ as a finite continued fraction in an even denominator form. 
The partial denominator $c_n$ is even if and only if the product $pq$ is
even. 
\end{lemma}  

\section{Converting the nonalternating denominator form into the even
  denominator form}
\label{sec:conv}

\begin{lemma}
\label{lem:2s}  
Assume there is an $i$ such that three consecutive partial denominators
$c_i$, $c_{i+1}$ and  $c_{i+2}$ in the nonsingular simple finite
continued fraction $[c_0,\ldots,c_n]$ have the same sign as
$\delta\in \{-1,1\}$. Then the continued fraction, obtained from $[c_0,\ldots,c_n]$ by the
following procedure, has the same evaluation as $[c_0,\ldots,c_n]$:
\begin{enumerate}
\item Replace $c_i$ with $c_i+\delta$ and $c_{i+2}$ with
  $c_{i+2}+\delta$.
\item Replace $c_{i+1}$ with $|c_{i+1}|-1$ copies of $2$. (In
  particular, simply remove $c_{i+1}$ if it is equal to $\delta$).
\item Keep the sign of $c_i+\delta$ and replace each subsequent
  sign in such a way that signs alternate up to and including the
  changed copy of $c_{i+2}$;
\item Replace $c_j$ with
  $(-1)^{|c_{i+1}|}\cdot c_j$ for each $j>i+2$.   
\end{enumerate}
\end{lemma}  
\begin{proof}
We proceed by induction on $|c_{i+1}|$.  
If $c_{i+1}=\delta$ then the statement may be obtained from 
Proposition~\ref{prop:e} using $-\delta$ and replacing 
$[\ldots c_i,\delta,c_{i+1}\ldots]=[\ldots c_i,\delta-\delta,\delta,-\delta-c_{i+1},-c_{i+2}\ldots]=[\ldots c_i+\delta,-\delta-c_{i+1},-c_{i+2}\ldots].$
Assume the statement is true for $|c_{i+1}|=k$ and
assume $|c_{i+1}|=k+1$. Let $\delta=\sign(c_{i+1})$.
Applying Proposition~\ref{prop:e} once yields
$$
[\ldots, c_i,c_{i+1}, c_{i+2}, \ldots, c_j,\ldots, c_n]
=
[\ldots, c_i+\delta,-\delta,\delta-c_{i+1}, -c_{i+2},
  \ldots, -c_j,\ldots, -c_n]. 
$$
Observe that in the resulting continued fraction, the consecutive
partial denominators $-\delta, \delta-c_{i+1}, $ $-c_{i+2}$ all
have the same sign as $-\delta$ and the absolute value of
$\delta-c_{i+1}$ is $k$. Applying the induction hypothesis to these
three consecutive partial denominators we obtain that our continued
fraction has the same evaluation as 
$$
[\ldots, c_i+\delta,-2\cdot \delta, 2\cdot\delta,\ldots,
  (-1)^{k} \cdot 2\cdot \delta, 
  (-1)^{k+1} (c_{i+2}+\delta), \ldots, (-1)^{k+1} c_j,\ldots,
  (-1)^{k+1} c_n].  
$$
\end{proof}

\begin{proposition}
\label{prop:pblock}
Assume a simple continued fraction $[c_0,\ldots,c_n]$ contains a
contiguous subsequence $c_m,c_{m+1},\ldots,c_{m+2k}$ of partial
denominators, satisfying the following conditions:
\begin{enumerate}
\item the integers $c_m,c_{m+1},\ldots,c_{m+2k}$ all have the same
  sign $\delta$;
\item $c_m$ and $c_{m+2k}$ are odd;
\item $c_{m+2i}$ is even for $1\leq i\leq k-1$.  
\end{enumerate} 
Then the continued fraction, obtained by the following
transformation, has the same evaluation:
\begin{enumerate}
\item replace $c_m$ with $c_m+\delta$ and $c_{m+2k}$ with
  $c_{m+2k}+\delta$;
\item for each  $i\in\{1,\ldots, k-1\}$ replace $c_{m+2i}$ with
  $c_{m+2i}+2 \delta$;
\item for each  $i\in\{1,\ldots, k\}$ replace $c_{m+2i-1}$ with
  $|c_{m+2i-1}|-1$ copies of $2$;
\item keep the sign of $c_m+\delta$ and replace each subsequent
  sign in such a way that signs alternate up to and including the
  changed copy of $c_{m+2k}$;
\item for each $j>m+2k+1$ we replace $c_j$ with
  $(-1)^{|c_{m+1}|+|c_{m+3}|+\cdots +|c_{m+2k-1}|}\cdot c_j$.   
\end{enumerate}  
\end{proposition}
\begin{proof}
To prove this we repeatedly use Lemma~\ref{lem:2s}. We explain the principle using an example given by the equality 
  $$
  [1,2,\mathbf{3}, 4, \mathbf{2}, 1, \mathbf{6}, 3,\mathbf{5},3]
  =
   [1,2,\mathbf{4}, -2, 2, -2, \mathbf{4}, -\mathbf{8}, 2, -2,\mathbf{6}, 3].
   $$
 In this example, $m=2$ and $k=3$. To help keep track of the strings of entries inserted between $c_2,c_4,c_6$ and $c_8$, these are marked in bold. We replace $c_3,c_5$ and $c_7$ by alternating strings of 2's. As a first step we apply Lemma~\ref{lem:2s} to
 $c_m, c_{m+1},c_{m+2}$. Thus we obtain that our continued fraction has the same
 evaluation as
 $$
 [\ldots, c_m+\delta,-2\delta, 2\delta, \ldots,
   (-1)^{|c_{m+1}|-1} 2\delta,   
   (-1)^{|c_{m+1}|} (c_{m+2}+\delta), \ldots, (-1)^{|c_{m+1}|}
   c_j, \ldots, (-1)^{|c_{m+1}|} c_n].  
 $$
 In our example
  $$
  [1,2,\mathbf{3}, 4, \mathbf{2}, 1, \mathbf{6}, 3,\mathbf{5},3]
  =
  [1,2,\mathbf{4}, -2, 2, -2, \mathbf{3}, 1, \mathbf{6}, 3,\mathbf{5},3].
   $$
  Next we apply Lemma~\ref{lem:2s} to
  $(-1)^{|c_{m+1}|} (c_{m+2}+\delta), (-1)^{|c_{m+1}|} c_{m+3},
  (-1)^{|c_{m+1}|} c_{m+4}$. This results in
  replacing $(-1)^{|c_{m+1}|} (c_{m+2}+\delta)$ with
  $(-1)^{|c_{m+1}|} (c_{m+2}+2\delta)$, $(-1)^{|c_{m+1}|} c_{m+3}$
  with a sequence of length $|c_{m+3}|-1$ in which copies of $2$ and $-2$
  alternate, and $(-1)^{|c_{m+1}|} c_{m+3}$ with $(-1)^{|c_{m+1}|+|c_{m+3}|}
  (c_{m+3}+\delta)$. All partial denominators after this one are
  multiplied by $(-1)^{|c_{m+3}|}$. In our example we obtain 
  $$
  [1,2,\mathbf{3}, 4, \mathbf{2}, 1, \mathbf{6}, 3,\mathbf{5},3]
  =
  [1,2,\mathbf{3}, -2,2,-2, \mathbf{4}, \mathbf{-7}, -3,\mathbf{-5},-3].
  $$
Note that $c_5=1$ is replaced with an empty sequence. We continue in a
similar fashion, applying Lemma~\ref{lem:2s} to the consecutive partial
denominators
$(-1)^{|c_{m+1}|+|c_{m+3}|+\cdots+|c_{m+2i-1}|} (c_{m+2i}+\delta)$,
$(-1)^{|c_{m+1}|+|c_{m+3}|+\cdots+|c_{m+2i-1}|} c_{m+2i+1}$ and 
$(-1)^{|c_{m+1}|+|c_{m+3}|+\cdots+|c_{m+2i-1}|} c_{m+2i+2}$ for
$i=1,2,\ldots,k-1$.
\end{proof}  

\begin{definition}
We say that a subsequence of consecutive partial denominators
$(c_m,\ldots,c_{m+2k})$ in a nonsingular finite simple continued
fraction $[c_0,\ldots, c_n]$ is a {\em primitive block} if it satisfies one 
of the following criteria: 
\begin{enumerate}
\item $k=0$ and $c_m$ is even;
\item $k\geq 1$ and $(c_m,\ldots,c_{m+2k})$ satisfies the
  hypotheses in Proposition~\ref{prop:pblock};
\item $k=0$, $m=n$ and $c_n$ can be odd or even.
\end{enumerate}
We say that a finite simple continued fraction $[c_0,c_1,\ldots,c_n]$
{\em has a primitive block decomposition} if the sequence of its partial
denominators may be written as a concatenation of primitive blocks. We
will say a primitive block is {\em trivial} if $k=0$, otherwise we say
it is nontrivial. We will also call a type (3) trivial primitive block
{\em exceptional}. 
\end{definition}

\begin{remark}
  \label{rem:convert}
{\em  
A primitive block decomposition, if it exists, allows us to use
Proposition~\ref{prop:pblock} to rewrite a 
continued fraction in even denominator form. Indeed, when we apply
Proposition~\ref{prop:pblock} to a nontrivial primitive block
$(c_m,\ldots,c_{m+2k})$, the sign of each $c_i$ satisfying $i<m$ remains
unchanged, and all $c_i$ satisfying $i>2m+1$ get multiplied by the same
power of $(-1)$. Hence the other primitive blocks of the continued
fraction remain primitive blocks, whereas the nontrivial primitive block
$(c_m,\ldots,c_{m+2k})$ is replaced by a concatenation of trivial
primitive blocks. Applying Proposition~\ref{prop:pblock} repeatedly we
may replace all nontrivial primitive blocks by a concatenation of
trivial primitive blocks. These contain even integers, except for the
last one, if that primitive block is exceptional.}
\end{remark}

\begin{example}
  \label{ex:pblock}
  {\em The rational number $1402/1813$ has the nonalternating simple
    continuous fraction expansion $[0, 1, 3, 2, 2, 3, 5, 1, 3]$. This has
    the primitive block decomposition $[0, 1, 3, 2, 2, 3; 5, 1, 3]$ (here primitive blocks are separated by semi-colons). By
    Proposition~\ref{prop:pblock}, we have
    $$
    [0, 1, 3, 2, 2, 3; 5, 1, 3]=
    [0, 2, -2, 2, -4, 2, -4; -5, -1, -3]=
    [0, 2, -2, 2, -4, 2, -4; -6, 4].
    $$

 }   
\end{example}  

\begin{proposition}
\label{prop:upbdec}  
If the primitive block decomposition of a finite simple continued fraction
$[c_0,\ldots,c_n]$ exists, then it is unique.   
\end{proposition}
\begin{proof}
Let $[c_0,\ldots,c_n]$ be a finite simple continued fraction with a primitive block decomposition.  We show this primitive block decomposition is unique.
The statement may be shown by induction on $n$. For $n=0$ the
  statement is obvious. For $n>0$ the partial denominator
$c_0$ is in a trivial primitive block by itself if and only if $c_0$ is
even. In that case we may apply the induction hypothesis to
$[c_1,\ldots,c_n]$. If $c_0$ is odd then it is contained in a nontrivial
primitive block that begins with $c_0$. By the definition of a
nontrivial primitive block, the right and of this block can be only at
$c_{2k}$ where $k$ is the least positive integer $i$ such that $c_{2i}$
is odd and of the same sign as $c_0$. We may apply the induction hypothesis to $[c_{2k+1},\ldots,c_n]$.
\end{proof}

\begin{remark}
  \label{rem:upbdec}
{\em The proof of Proposition~\ref{prop:upbdec} amounts to providing a
recursively defined algorithm that finds the unique primitive block
decomposition given that such a decomposition exists.
Assume that $[c_0,\ldots,c_n]$ is a nonalternating simple continuous fraction expansion:
\begin{enumerate}
\item If $c_0$ is even, put it into a trivial primitive
  block, and proceed recursively on $[c_1,\ldots,c_n]$.
\item If $c_0$ is odd and $n=0$ then $c_0$ is in an exceptional trivial
  primitive block.
 \item If $c_0$ is odd and $n>0$ then the nontrivial primitive block
   containing $c_0$ ends at the first odd $c_{2k}$, and we proceed
   recursively on  $[c_{2k+1},\ldots,c_n]$.
 \end{enumerate} 
}
\end{remark}  

Next we need to show that for a rational number $p/q$ we can construct a primitive block
decomposition.

\begin{theorem}
\label{thm:noepb}  
Every nonzero rational number $p/q$ may be written in two ways 
as a finite simple continued fraction in a nonalternating form. Exactly
one of of these nonalternating forms has a primitive block
decomposition. This primitive block decomposition contains no
exceptional trivial primitive block if and only if $pq$ is even.
\end{theorem}  
\begin{proof}
The first sentence recalls Lemma~\ref{lem:nonaltu}. What we need to show
is that exactly one of the nonalternating forms has a primitive block
decomposition. Note that, for $p/q=0$ the only way to write it in
nonalternating form is $p/q=[0]$ where $0$ is a trivial primitive
block.

From now on, without loss of generality, we may assume $p/q>0$. Such a
continued fraction can be written in a nonalternating form in exactly
two ways: as $[c_0,c_1,\ldots, c_n]$ satisfying $c_0\geq 0$, $c_1\cdots
c_n >0$ and $c_n>1$, or as $[c_0,c_1,\ldots, c_{n-1},c_n-1,1]$.

{\bf\noindent Case 1:} $[c_0,c_1,\ldots, c_n]$ has a primitive block
decomposition. By Proposition~\ref{prop:upbdec} this decomposition is
unique.
We need to show that $[c_0,c_1,\ldots, c_{n-1},c_n-1,1]$ cannot have a primitive block
decomposition.
If $c_n$ is even then it is contained in a trivial primitive
block by itself and $[c_0,c_1,\ldots, c_{n-1}]$ has a primitive block
decomposition with no exceptional block at $c_{n-1}$.
Then the partial denominator
$c_n-1$ is an odd number that has to be the left end of a nontrivial
primitive block, that has no right end in $[c_0,c_1,\ldots, c_{n-1},c_n-1,1]$. 
If $c_n$ is odd then there are two choices. Either $c_n$ is the right end of a block that is not exceptional or $c_n$ is an exceptional block. In either case the last block in $[c_0,c_1,\ldots, c_{n-1},c_n-1,1]$ does not have the correct format and  $[c_0,c_1,\ldots, c_{n-1},c_n-1,1]$ cannot have a primitive block decomposition. 

{\bf\noindent Case 2:} $[c_0,c_1,\ldots, c_n]$ has no primitive block
decomposition. We need to show that $[c_0,c_1,\ldots, c_{n-1},$ $c_n-1,1]$ has a primitive block
decomposition. 
The algorithm described in
Remark~\ref{rem:upbdec} applied to $[c_0,c_1,\ldots, c_n]$ finds an $m<n$ such that
$[c_0,\ldots,c_{m-1}]$ has a primitive block decomposition with no
exceptional trivial primitive block at the end, $c_{m}$ is odd but
$c_{m+2i}$ is even for all $m+2i\leq n$. If $n-m$ is even, say $n=m+2k$
then $c_m, c_{m+1}, \ldots, c_{m+2k}-1$ is a nontrivial primitive block
in $[c_0,\ldots, c_{n-1}, c_n-1,1]$ and the last partial denominator $1$
belongs to an exceptional trivial primitive block. If $n-m$ is odd, say
$n=m+2k+1$ then $c_m, c_{m+1}, \ldots, c_{m+2k+1}-1, 1$ is a nontrivial
primitive block in $[c_0,\ldots, c_{n-1}, c_n-1,1]$.    

So far we have shown that every rational number may be uniquely written
into a nonalternating form that has a primitive block decomposition. The
second half of the statement is a direct consequence of
Lemma~\ref{lem:evenden} and Remark~\ref{rem:convert}. 
\end{proof}

\begin{example}
  \label{ex:pblock2}
  {\em As seen in Example~\ref{ex:pblock}, the nonalternating simple continuous fraction expansion of $1402/1813$ is $[0, 1,
      3, 2, 2, 3, 5, 1, 3]$, which has 
    the primitive block decomposition $[0, 1, 3, 2, 2, 3; 5, 1, 3]$. The
    other nonalternating simple continued fraction decomposition is
    $[0, 1, 3, 2, 2, 3, 5, 1, 2, 1]$, which has no primitive block
    decomposition: the algorithm described in Remark~\ref{rem:upbdec}
    would yield the blocks $0$ and $1, 3, 2, 2, 3,$ and the block
    starting at $5$ is incomplete. 

 }   
\end{example}

\section{Transforming rational links into all-even form}
\label{sec:transf}

We encode an unoriented rational link diagram by a continued fraction
$p/q=[0,a_1,a_2,\ldots,a_n]$ whose evaluation has absolute value at most one
and satisfies $a_1\cdots a_n\neq 0$, where the numbers $|a_1|,\ldots, |a_n|$ are
the numbers of consecutive half-turn twists in the twistboxes $B_1,
\ldots, B_n$ following the sign convention as shown in Figure~\ref{fig:rationallink} (without the 
orientation of the strands). 
\begin{figure}[h]
%60%  
\begin{center}
\input{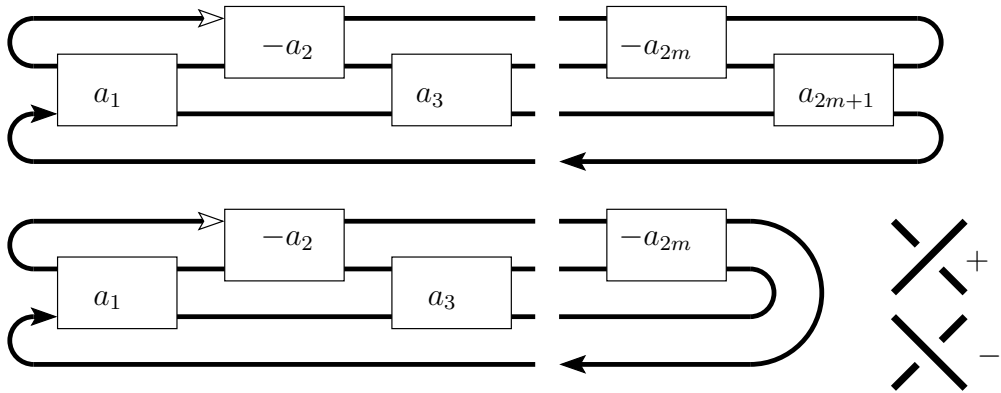}
\end{center}
\caption{Without the orientation: the sign convention used to define the standard form of non-oriented rational links; With orientation: preferred standard form of oriented rational links.}
\label{fig:rationallink}
\end{figure}
We denote the link presented in this form by the symbol $b(q,p)$ or by the vector $(a_1,a_2,\ldots,a_n)$, and call such a diagram a {\em standard diagram} of the unoriented rational link.
Notice that the left end of the diagram is fixed, the closing on the right end
depends on the parity of $n$. Our choice of twisting
signs agrees with~\cite{BZ, Cromwell}, but some articles use the mirror image~\cite{Du}. Since the braid index is an invariant of oriented links, we will need to consider the rational links with orientation. Once a rational link $b(q,p)$ is given an orientation, the crossings in it also have signs known as {\em crossing sign} in knot theory given by the convention shown in Figure \ref{crossingsign}, which is not to be confused with the signs of the half-turn twists in the twistboxes in Figure~\ref{fig:rationallink}. 
\begin{figure}[h] 
\begin{center}
\includegraphics[scale=0.5]{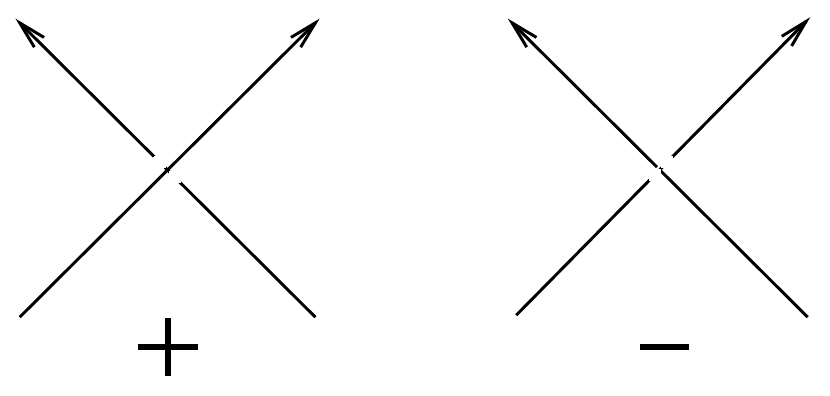}
\end{center}
\caption{The crossing sign convention at a crossing in an oriented link diagram.}
\label{crossingsign}
\end{figure}
In order to avoid this confusion, in the rest of this paper, we will
call the sign given in Figure~\ref{fig:rationallink} the {\em twist
  sign} and the sign given by Figure~\ref{crossingsign}  the  {\em
  crossing sign} (its usual name in knot theory). We will denote the
twist sign of the crossings in twistbox $B_i$ (corresponding to $a_i$ in
$b(q,p)=(a_0,a_1,...,a_n)$) by $\signt(a_i)$ and denote the crossing
sign of these crossings by $\signc(a_i)$. Note that the twist sign is
given by the formula
\begin{equation}
\signt(a_i)=(-1)^{i-1}\sign(a_i). 
\end{equation}
If there is a need to indicate the crossing sign and the twist sign of the crossings corresponding to an $a_i$ entry in a twistbox simultaneously, then we will indicate the twist sign by placing a $+$ or $-$ in front of the number $a_i$ and indicate the crossing sign by placing a superscript $+$ or $-$ to $a_i$. An example is given in Figure \ref{crossing_example}. Notice that in general, a rational link diagram in its standard form is not necessarily alternating. In fact, $b(q,p)=(a_0,a_1,...,a_n)$ is an alternating diagram if and only if the continued fraction expansion $[a_0,a_1,...,a_n]$ is in nonalternating denominator form. From this it follows that every rational link has an alternating link diagram, see~\cite{BZ, Cromwell}.

\begin{figure}[h] 
\begin{center}
\includegraphics[scale=0.8]{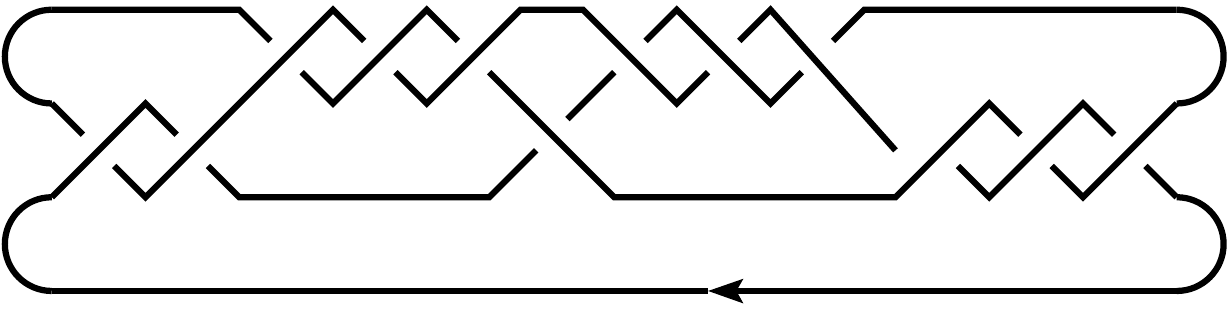}
\end{center}
\caption{The crossing sign convention for crossings in an oriented link diagram: the signs in front of the $a_j$'s indicate which strand is on top as indicated in Figure~\ref{fig:rationallink}, while the $+$ or $-$ in the superscript indicate the crossing sign given by Figure~\ref{crossingsign}. Shown is the example $[0,2^-,-3^+,-1^-,3^+,3^+]$.}
\label{crossing_example}
\end{figure}

Since we are interested in the braid index and the HOMFLY polynomial of
a rational link we need to consider rational links as oriented links. 
The classification of these links can be found in~\cite{BZ} and is due to Schubert~\cite{Schubert}.
\begin{theorem}
\label{classificaitonrationallinks}  $b(q,p)$ and $b(q^\p,p^\p)$ are equivalent as oriented links if and only if
$q=q^\p$ and $p^{\pm 1}\equiv p^\p \mod(2q)$, and are equivalent as unoriented links if and only if
$q=q^\p$ and $p^{\pm 1}\equiv p^\p \mod(q)$.
\end{theorem}

Rational links are invertible and therefore we will orient the
lowermost strand from the right to the left (as indicated by solid arrows)
without loss of generality. If the orientation of the top strand at left
side of the link diagram in its standard form is as shown in
Figure~\ref{fig:rationallink} (indicated by the hollow arrow), then we
say that the link diagram is in a {\em preferred standard form} and we will denote
by $\tilde{b}(q,p)$ the corresponding link $b(q,p)$ if we want to emphasize that $b(q,p)$ is given by diagram is in preferred standard form. It is important
to note that the braid index formula obtained in~\cite{Mu}
(Proposition~\ref{prop:CM}) is based on preferred standard forms of
rational links. As to the HOMFLY polynomial, the result of Lickorish and
Millett~\cite[Proposition 14]{Li-Mi} we use is stated in terms of the
Conway notation~\cite{Co}, but we rephrased it in
Proposition~\ref{prop:LM} below as a statement on link diagrams in
preferred standard form.    

Theorem \ref{classificaitonrationallinks} needs some more explanation for oriented links. In the case that $b(q,p)$ has two components, we have two choices for the orientation of the second component (namely the one that does not contain the bottom strand of the diagram), and one and only one such choice would produce a preferred standard diagram $\tilde{b}(q,p)$. 
By the procedure shown in Figure~\ref{fig:preferredform}, we can convert
any rational link (or knot) that is not in preferred standard form to a preferred standard form
representation: first we fold up the part shown with a dashed line in
the upper right corner and then we fold down the entire diagram,
applying a spatial $180$ degree rotation that changes the twist sign of all
crossings. In Figure~\ref{fig:preferredform} the continued fraction representing the original link
(disregarding the orientation) is $5/18=[0,3,1,1,2]$, after the
transformation we obtain the link represented by the continued fraction
$-13/18=[0,-1,-2,-1,-1,-2]$. The second continued fraction is obtained from
the first by replacing $3$, the first nonzero partial denominator, with
the sequence $1,2$ and then taking the negative of all partial
denominators. We note that $p=-5$ and $p'=-13$ do not satisfy $p^{\pm 1}\equiv p^\p \mod(2q)$.
We resolve this problem by assuming that in the case of oriented rational links we are always given the preferred standard form $\tilde{b}(q,p)$ by a vector $(a_1,a_2,\ldots,a_n)$. That is the link $\tilde{b}(18,5)$ given by $5/18=[0,3,1,1,2]$ must have the second component oriented differently than shown in Figure~\ref{fig:preferredform}.
Given $\tilde{b}(q,p)$ the other choice of orientation produces $\tilde{b}(q,-(q-p))$, which is usually a different link type than $\tilde{b}(q,p)$ by Theorem~\ref{classificaitonrationallinks}.  In the example of Figure~\ref{fig:preferredform} $\tilde{b}(q,-(q-p))=\tilde{b}(18,-13)$ is shown on the bottom right.

This observation may be generalized as follows.
\begin{lemma}
\label{lem:flip}  
If $p/q=[0,c_1,c_2,\ldots,c_n]$ satisfies $|c_1|>1$ and $c_1\cdots c_n
\neq 0$
then
$1-p/q=(q-p)/q=[0,\sign(c_1),c_1-\sign(c_1),c_2,\ldots,c_n]$.  
Moreover if $\tilde{b}(q,p)$ is given by the vector $(c_1,c_2,\ldots,c_n)$ then the vector $(\sign(c_1),c_1-\sign(c_1),c_2,\ldots,c_n)$ gives the mirror image of $\tilde{b}(q,-(q-p))$.
\end{lemma}
Note that the operation $p/q\mapsto 1-p/q$ is an involution and that the
involution described in Lemma~\ref{lem:flip} above is somewhat similar to
the involution described in Lemma~\ref{lem:nonaltu}. 

\begin{figure}[h] 
\begin{center}
\includegraphics[scale=0.6]{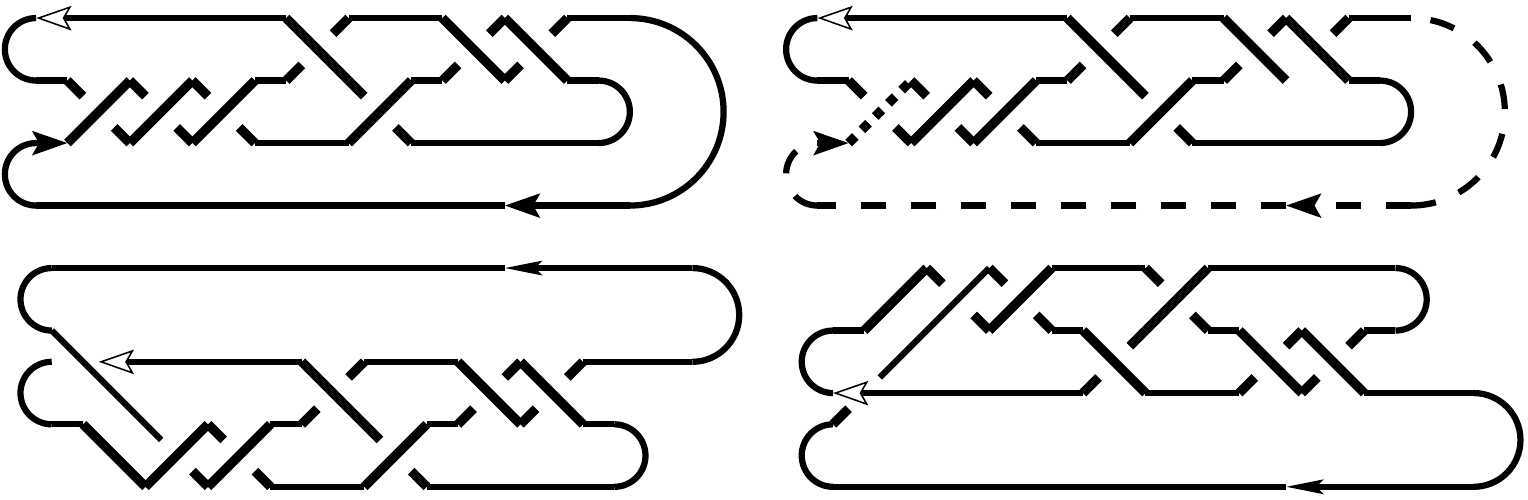}
\end{center}
\caption{The procedure of changing a two component rational link in a non-preferred standard form to a preferred standard form.}
\label{fig:preferredform}
\end{figure}

Define the sign of a twistbox $B_i$ as the crossing sign of all
crossings in it and denote it (by abuse of notation) by $\signc(B_i)$. Under the assumption that $b(q,p)$ is in a preferred standard form and that $a_i>0$ for all $i$ (hence $\signc(B_1)=+1$), we note that the signs of adjacent twistboxes are related by the automaton shown in Figure~\ref{fig:automaton}. (If $a_i>0$ but $\signc(B_1)=-1$ then the diagram is not in a preferred standard form and we would need to change to $\tilde{b}(q,q-p)$.) Inspecting the signs of the crossings in the states of the automaton, we obtain the following result.

\begin{figure}[h] 
\begin{center}
\input{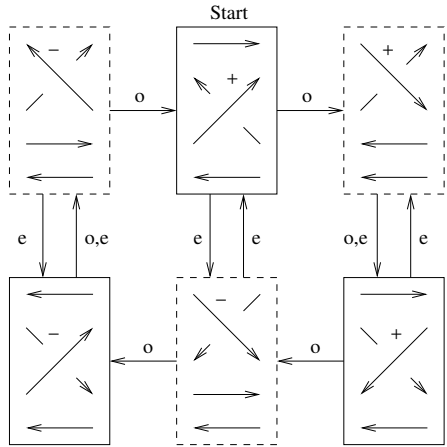}
\end{center}
\caption{Automaton, parsing the signs of crossings in an alternating
  link. $o$ and $e$ stand for an odd or even number of crossings in a twistbox respectively. }
\label{fig:automaton}
\end{figure}

\begin{theorem}
\label{thm:blockdec}  
Suppose $\tilde{b}(q,p)$ is
represented by a nonalternating 
continued fraction $p/q=[0,a_1,$ $\ldots,a_n]$
that has a primitive block decomposition with no exceptional primitive
block. Then twistboxes associated with the same primitive block have the same sign, while the twistboxes  associated with adjacent primitive blocks have opposite signs.
\end{theorem}  
\begin{proof}
Let the diagram $D$ be given by the vector $(a_1,\ldots,a_n)$.
(i) Let $(a_1,\ldots,a_{1+2k})$ be the first primitive block.
Without loss of generality we may assume 
$a_i>0$ for all $i$, the case $a_i<0$ for all $i$ is completely
analogous. This means that $B_1$ is involving the middle two strings of the preferred standard diagram as shown in top middle of Figure~\ref{fig:automaton}.
If $a_1$ is even then it is the only element of its primitive block ($k=0$) and $\signc(B_{2})=-1$. The partial
denominator $a_{2}$ corresponds to the bottom center state of Figure~\ref{fig:automaton} and therefore $\signc(a_2)=-1$. If $a_1$ is
odd, then it is the first element of a nonsingleton primitive block. We have
$\signc(B_{2})=1$, $a_2$ corresponds to the top right state in the
automaton and $\signc(a_2)=1$. The subsequent partial denominators correspond to twistboxes
filled with positive crossings, as long as we oscillate between the top
right and the bottom right states of Figure~\ref{fig:automaton}, and we reach the first negative
state exactly when we leave the bottom right state moving left. This
corresponds exactly to finishing the primitive block containing $a_{1+2k}$ with $a_{1+2k})$ odd and $\signc(B_{2+2k})=-1$ corresponds to the bottom center state of Figure~\ref{fig:automaton}.

(ii) Let $(a_m,\ldots,a_{m+2k})$ be the second primitive block starting at the bottom center state of Figure~\ref{fig:automaton}.
If $a_m$ is even then it is the only element of its primitive block ($k=0$) and $\signc(B_{m+1})=1$. The partial
denominator $a_{m+1}$ corresponds to the top center state of Figure~\ref{fig:automaton}. If $a_m$ is
odd, then it is the first element of a nonsingleton primitive block, we have
$\signc(B_{m+1})=-1$ and $a_{m+1}$ corresponds to the bottom left state in the
automaton. The subsequent partial denominators will all label twistboxes
filled with negative crossings, as long as we oscillate between the bottom
left and the top left states of Figure~\ref{fig:automaton}, and we reach the first positive
state exactly when we leave the top left state moving right. This
corresponds exactly to finishing the primitive block containing $a_{m+2k}$ with $a_{m+2k})$ odd and $\signc(B_{m+2k+1})=1$ corresponds to the top center state of Figure~\ref{fig:automaton}.

The theorem follows by induction on the number of blocks. 
\end{proof}  

\begin{theorem}
\label{thm:blocktrans}  
Suppose $\tilde{b}(q,p)$ is
represented by a nonalternating 
continued fraction $p/q=[0,a_1,\ldots,$ $a_n]$
that has a primitive block decomposition with no exceptional primitive
block. For each $i\in\{1,\ldots,n\}$ define $\tau_(i)$ by
$$
\tau(i)=i-1-\mid\{j\leq i \::\: \signc(a_j)\neq \signc(a_{j-1})\mbox{ or $j=1$
  and $a_1$ is even}\}\mid+\sum_{{j<i \atop
    \signc(a_j)=\signc(a_{j-1})=(-1)^j}} (a_j-2).
$$
Then the all-even continued fraction representation of $p/q$ may be
obtained by replacing each $a_i$ as follows.
\begin{enumerate}
\item If $a_1$ is even, keep it, if it is odd, replace it with
  $a_1+\sign(a_1)$.   
\item If $\signc(a_i)\neq \signc(a_{i-1})$ then replace $a_{i}$ with
  $(-1)^{\tau(i)}\cdot a_i$ if $a_i$ is even and with $(-1)^{\tau(i)}\cdot
  (a_i+\sign(a_1))$ if $a_i$ is odd.
\item If $\signc(a_i)=\signc(a_{i-1})=(-1)^{i-1}\cdot \sign(a_1)$ then
  replace $a_i $ with $(-1)^{\tau(i)}\cdot
  (a_i+\sign(a_1))$ if $a_i$ is odd and with $(-1)^{\tau(i)}\cdot
  (a_i+2\cdot \sign(a_1))$ if $a_i$ is even.
\item If $\signc(a_i)=\signc(a_{i-1})=(-1)^{i}\cdot \sign(a_1)$ then  
  replace $a_i $ with the sequence
$$(-1)^{\tau(i)}\cdot \sign(a_1) \cdot 2,
(-1)^{\tau(i)}\cdot \sign(a_1) \cdot (-2),\ldots (-1)^{\tau(i)}\cdot \sign(a_1) \cdot
(-1)^{|a_i|-2}\cdot 2$$ of length $|a_i|-1$. 
\end{enumerate}
\end{theorem}

\begin{example}{\em
Using the example of Proposition~\ref {prop:pblock} we have the link 
$\tilde{b}(49654, 34651)$ given by 
$ [0,+1^+, +2^+, +3^+, +4^-, +2^+, +1^-, +6^-, +3^-, +5^+, +3^+]$.
Applying the procedure given by Theorem~\ref{thm:blocktrans}  we have $\sign(a_1)=+1$, $\tau(2)=1,\tau(3)=\tau(4)=\tau(5)=\tau(6)=2, \tau(7)=3, \tau(8)=8, \tau(9)=9, \tau(10)=9$. This will result in the even denominator form
$[0, 2, -2, 4, 4, 2, 2, -2, 2, -2, 2, -2, 4, 6, -2, $ $2]$.}
\end{example}

\begin{proof}
The states of the automaton shown in Figure~\ref{fig:automaton} indicate
the position of each $a_i$ in its primitive block: the two middle states
correspond to a last element of a primitive block, the upper right and
lower left states correspond to $a_i$ being in an odd (but not last)
position of a block of odd length, the lower right and upper left states
correspond to $a_i$ being in an even position in a block of odd
length. (We associate to each $a_i$ the state reached {\em after}
reading $a_i$.) These states are also identifiable if we know the crossing
signs in the twistboxes associated to $a_{i-1}$ and $a_i$ and we also
know the parity of $a_i$. The stated rules then follow from
Proposition~\ref{prop:pblock}, after verifying that
$(-1)^{\tau(i)}\sign(a_1)$ is the correct sign of the number (or first
number in the sequence) obtained by transforming $a_i$. Indeed, the
transformed numbers must be alternating in sign, except when the
crossing sign of some $a_j$ is different from the crossing sign of
$a_{j-1}$: in this case the first number obtained by transforming
$a_{j+1}$ must have the same sign as the last number obtained by
transforming $a_j$.   

\end{proof}

\section{The braid index of rational links}
\label{sec:braid}

In this section we show how Theorem~\ref{thm:blockdec} may be used to
derive the formula for the braid index of an alternating rational link,
given in~\cite{DL}, using the result of Murasugi~\cite[Proposition
  10.4.3]{Cromwell} which allows one to compute the braid index of an
oriented rational link $K$ that has a preferred standard form. 

\begin{definition}
  \label{def:CM}
Let $p/q$ be a rational number such that $pq$ is even. Assume its unique
even denominator form continued fraction expansion is
$[2d_0,2d_1,\ldots,2d_n]$. We define the {\em Cromwell-Murasugi
  index}  of $p/q$ to be the number
$\sum_{i=0}^n |d_i| -t+1$ where $t$ is the number of indices $i$ such
that $d_i d_{i+1}<0$. 
\end{definition}  

In terms of the above definition, Murasugi's result~\cite[Proposition
  10.4.3]{Cromwell} may be stated as follows. 

\begin{proposition}
\label{prop:CM}  
If an oriented rational link $K$ has a preferred standard form
$\tilde{b}(q,p)$, then its braid index is given by the Cromwell-Murasugi
index of $p/q$, as defined in Definition~\ref{def:CM}. 
\end{proposition}

The next theorem shows how to compute the braid index of a rational link
$K$ that has a preferred standard form, using the unique continued fraction
expansion that corresponds to a an alternating link and has a primitive
block decomposition. 

\begin{theorem}
\label{thm:CMpblock}
  Suppose $p/q$ is a rational number such that $pq$ is even. Let
$[c_0,\ldots,c_n]$  be the unique nonalternating continued fraction
expansion of $p/q$ that has a primitive block decomposition. Then the
Cromwell-Murasugi index of $p/q$ may be computed by adding $1$ to the
sum of all $|c_i|/2$ such that $c_i$ is in an odd position in the
primitive block containing it.  
\end{theorem}   
\begin{proof}
Consider a primitive block $[c_m, c_{m+1}, \ldots, c_{m+2k}]$. After
transforming the preceding primitive blocks into even denominator form,
this block turns into $[c_m\cdot\delta , c_{m+1}\cdot\delta,
  \ldots, c_{m+2k}\cdot\delta]$ for some $\delta
\in\{-1,+1\}$. Note that no sign change occurs between blocks. We use
the procedure described in Lemma~\ref{lem:2s} to transform this block
into even denominator form. For each $i\in \{1,\ldots,k\}$, the
partial denominator $c_{m+2i-1}\cdot \delta$ is replaced with
a sequence of $2$'s of length $|c_{m+2i-1}|-1$. Hence the total length of
$2k+1$ of the primitive block is increased to
$$
2k+1+\sum_{i=1}^k (|c_{m+2i-1}|-2)=1+\sum_{i=1}^k |c_{m+2i-1}|.
$$
Since signs in the transformed block alternate, the total number of sign
changes in the transformed block is
$\sum_{i=1}^k |c_{m+2i-1}|$. The sum of the halves of the absolute values of the
partial denominators in the transformed block is 
\begin{eqnarray*}
&&\frac{1}{2}\left((|c_m|+1) + (|c_{m+2k}|+1) + \sum\limits_{i=1}^{k-1} (|c_{m+2i}|+2)+
  2\sum\limits_{i=1}^k   (|c_{m+2i-1}|-1)\right)\\
  &=&
\frac{1}{2}\sum\limits_{i=0}^k |c_{m+2i}|+\sum_{i=1}^k  |c_{m+2i-1}|.
\end{eqnarray*}
Subtracting the number of sign changes yields
$\left(\sum\limits_{i=0}^k |c_{m+2i}|\right)/2$.
\end{proof}  

\begin{example}
  \label{ex:pblock3}
  {\em As seen in Example~\ref{ex:pblock}, the rational number
    $1402/1813$ has the nonalternating simple 
    continued fraction expansion $[0, 1, 3, 2, 2, 3; 5, 1, 3]$ that has
    a primitive block decomposition. By Theorem~\ref{thm:CMpblock}, the
    Cromwell-Murasugi index of $1402/1813$ is
    $$
1+\frac{1+2+3}{2}+\frac{5+3}{2}=8. 
    $$
  We have also seen in Example~\ref{ex:pblock} that the continued
  fraction expansion of $1402/1813$ is\\
  $[0, 2, -2, 2, -4, 2, -4; -6,
    4]$. By definition, the Cromwell-Murasugi index of $1402/1813$ is
  $$
(0+1+1+1+2+1+2+3+2)-6+1=13-6+1=8.
  $$
}   
\end{example}

If $b(q,p)$ is not in the preferred standard form, then the above theorem cannot be applied to the diagram $b(q,p)$ directly. That is, in order to use Murasugi's formula to calculate the braid index of $b(q,p)$, we have to apply the formula to $\tilde{b}(q,-(q-p))$ (or $\tilde{b}(q,q-p)$ since $\tilde{b}(q,q-p)$ is the mirror image of $\tilde{b}(q,-(q-p))$ and has the same braid index). For example $b(4,1)=[0,4]$ (a $(4,2)$ torus link) with the orientation given in Figure~\ref{fig:rationallink} yields the link $\tilde{b}(4,1)$, which has linking number $2$ and braid index $3$. If we reorient the top  component in $b(4,1)$, and change it to the preferred standard form as shown by the procedure illustrated in Figure~\ref{fig:preferredform}, then it becomes $\tilde{b}(4,-3)=[0,-1,-3]$, which has linking number $-2$ and braid index $2$. Passing to its mirror image $\tilde{b}(4,3)=[0,1,3]$ will lead us to the same braid index $2$, although it changes its linking number to $2$. When the rational link is a knot (that is, a link with one component), the orientation of the knot is completely determined by the orientation of the bottom strand and the knot diagram may or may not be in a preferred standard form. Specifically, $b(q,p)$ is in a preferred standard form if and only if $p$ is even ($q$ is odd since $b(q,p)$ is a knot). Thus, if $pq$ is odd, then $\tilde{b}(q,p)$ does not exist and one needs to use the same procedure shown in Figure~\ref{fig:preferredform} to change it to $\tilde{b}(q,p-q)$ in order to apply Murasugi's formula. We now establish an alternative method to compute the Cromwell-Murasugi index, using only facts about different ways to expand a rational number into continued fractions.

%As we have mentioned earlier, In such a case the diagram is equivalent to $\tilde{b}(q,-(q-p))$ as shown by the procedure illustrated in Figure~\ref{fig:preferredform}. Since passing to the mirror image does not change the braid index, we may then use $\tilde{b}(q,q-p)$ instead. 

\begin{theorem}
\label{thm:braid}
Assume that a rational link $K$ has a preferred standard form diagram that is
represented by the alternating continued fraction
$p/q=[0,a_1,\ldots,a_n]$ where all $a_i>0$, then the braid index of $K=\tilde{b}(q,p)$ is given by
\begin{equation}\label{mid_formula}
\mathbf{b}(K)=1+\frac{1}{2}\left( \sum\limits_{j\geq 0, \signc(B_{2j+1})=\signc(B_1)} a_{2j+1}
  +\sum\limits_{j\geq 1, \signc(B_{2j})=-\signc(B_1)} a_{2j}\right)+c(p/q,n).
\end{equation}
Here the correction term $c(p/q,n)$ is given by 
$$
c(p/q,n)=
\begin{cases}
0 & \mbox{if $n$ is odd and $\signc(B_n)=\signc(B_1)$, or $n$ is even and
  $\signc(B_n)=-\sign(B_1)$;}\\
1/2 & \mbox{if $n$ is odd and $\signc(B_n)=-\signc(B_1)$, or $n$ is even and
  $\signc(B_n)=\sign(B_1)$.}\\
\end{cases}  
$$
\end{theorem}  

\begin{remark}{\em
Since $K$ is in a preferred standard form, it is necessary that $pq$ is even. It is important that we apply the theorem to the preferred standard form only. For example for the knot $10_{37}=b(53,30)=b(53,23)$, 
$30/53=[0,1^+,1^+,3^+,3^-,2^-]$ is in the preferred standard form while $23/53=[0,2^-,3^-,3^+,2^+]$ is not. Consequently, if we apply Theorem~\ref{thm:braid} to $30/53=[0,1^+,1^+,3^+,3^-,2^-]$ we obtain the correct braid index of five, however if we use $23/53=[0,2^-,3^-,3^+,2^+]$ then the formula in Theorem~\ref{thm:braid} yields an incorrect braid index of three.}
\end{remark}

\begin{proof}
Since $pq$ is even, it has a unique nonalternating continued fraction
expansion that has primitive block decomposition, and this primitive
block decomposition contains no exceptional primitive block by
Theorem~\ref{thm:noepb}. 

{\bf\noindent Case 1:} $[0,a_1,\ldots,a_n]$ is the nonalternating
continued fraction expansion of $p/q$ that has a primitive block
decomposition. By Theorem~\ref{thm:blockdec}, the signs of all
crossings are the same in all twistboxes labeled by partial denominators
associated to the same primitive block, and they are opposite for crossings in
twistboxes associated to adjacent primitive blocks. 

By Theorem~\ref{thm:CMpblock}, computing the braid index of $K$ is one
more than the sum of all $a_i/2$ such that $a_i$ is in an odd position in
the block containing it. For the block containing $a_1$ this formula
calls for summing over all odd indices $i$ in the block. In the next
block, the sign of the crossings is opposite and the entries $a_i$ are the even
indexed entries. Since all primitive blocks have an odd number of
entries, the same pattern continues all the way: we have to take the sum
of all $a_{2i+1}/2$ such that the sign of the crossings in twistbox
number $2i+1$ is the same as the sign of the crossings in the
twistbox number $1$, and we have to add the sum
of all $a_{2i}/2$ such that the sign of the crossings in the twistbox
number $2i$ is the opposite of the sign of the crossings in the
twistbox number $1$. In this case, we want the correcting term
$c(p/q,n)$ to equal zero. The sign of all crossings in twistbox number
$n$ is the same as the sign of all crossings in the first twistbox
exactly when the number of primitive blocks containing
$a_1,\ldots,a_n$ is odd, which is equivalent to $n$ being odd.

{\bf\noindent Case 2:} $[0,a_1,\ldots,a_n]$ is the other nonalternating
continued fraction expansion of $p/q$, the one that has no primitive
block decomposition. As seen in the proof of Theorem~\ref{thm:CMpblock},
the nonalternating continued fraction expansion of $p/q$ that has a
primitive block expansion is $[0,a_1,\ldots,a_{n-1}, a_{n}-1,1]$ if
$a_n>1$ and it is $[0,a_1,\ldots,a_{n-2}, a_{n-1}+1]$ if
$a_n=1$.

{\bf\noindent Subcase 2a:} We assume $a_n>1$. In this case the primitive
block decomposition of $[0,a_1,\ldots,a_{n-1}, a_{n}-1,1]$ ends with a
nontrivial block containing $a_{n+1}:=1$ (by Theorem~\ref{thm:CMpblock}) and the sign of the crossing in
twistbox number $n+1$ is the same as the sign of all crossings in
twistbox number $n$. This sign is the same as $\signc(B_1)=\signc(B_n)$ if $n+1$ is
odd, and it is $-\signc(B_1)=\signc(B_n)$ if $n+1$ is even. So the partial denominator
$a_{n+1}=1$ contributes $1/2$ to the braid index. This is the only term
not included in the sum
$$
1+\frac{1}{2}\left(\sum\limits_{j\geq 0, \signc(B_{2j+1})=\signc(B_1)} a_{2j+1}
  +\sum\limits_{j\geq 1, \signc(B_{2j})=-\signc(B_1)} a_{2j}\right)
$$
computed using $[0,a_1,\ldots,a_n]$. Note that for $[0,a_1,\ldots,a_n]$ or for  $[0,a_1,\ldots,a_n-1,1]$ the term
$a_n/2$ or $(a_n-1)/2$ does not appear in our formula. The correct
braid index is thus obtained by adding the correcting term $c(p/q,n)=1/2$ to the contribution of the $a_i$s. 
%We also have that $n$ is odd and $\signc(B_1)=-\signc(B_n)$, or $n$
%is even and $\signc(B_1)=\signc(B_n)$. 

{\bf\noindent Subcase 2b:} We have $a_n=1$. In this case
$[0,a_1,\ldots,a_{n-2}, a_{n-1}+1]$ has a primitive block
decomposition. We either have that $\signc(B_1)=\signc(B_{n-1})$ and $n-1$ is
odd, or that $\signc(B_1)=-\signc(B_{n-1})$ and $n-1$ is even. 
%In the link diagram encoded by $[0,a_1,\ldots,a_n]$
%the same statements hold, and in addition $\signc(B_n)=-\signc(B_{n-1})$
Hence $a_n/2$ in not included in the sum 
$$
1+\frac{1}{2}\left(\sum\limits_{j\geq 0, \signc(B_{2j+1})=\signc(B_1)} a_{2j+1}
  +\sum\limits_{j\geq 1, \signc(B_{2j})=-\signc(B_1)} a_{2j}\right)
$$
and the contribution $a_{n-1}/2$ needs to be increased by $1/2$ to
obtain the correct braid index.  
\end{proof}  

\begin{remark}
\label{mirrorinvariance}{\em
The formula of Theorem~\ref{thm:braid} is invariant under a change to a mirror image. If we compare two mirror image rational link diagrams given by the vectors $[0,a_1,a_2,\ldots,a_n]$ and $[0,-a_1,-a_2,\ldots,-a_n]$ then all signs $\signc(B_i)$ change but this has no effect on the result in the formula of Theorem~\ref{thm:braid}. }
\end{remark}

\section{Connecting our braid index formulas}
\label{sec:connect}
 
Let us now present the formula derived in~\cite{DL}, starting with the convention used there. Let $p$, $q$ be a pair of positive and co-prime integers such that $q> p$ and consider the minimal (alternating) link diagram obtained from the expansion $p/q=[0,a_1,a_2,\ldots,a_n]$ with positive entries $a_1\cdots a_n\neq 0$. The ``standard form" adopted in~\cite{DL} is similar to the link diagram given in Figure~\ref{fig:rationallink} with the following differences: (i) it is based on an expansion $p/q=[0,a_1,a_2,\ldots,a_n]$ with positive entries $a_1\cdots a_n\neq 0$; (ii) it is required that $n=2k+1$ be odd: this can always be done since if $n$ is even then it is necessary that $a_n>1$ and we can move to $p/q=[0, a_1,\ldots,a_n- 1,1]$; 
(iii) the orientation of the bottom left strand is from right to left, but there is no restriction on the orientation of the top left strand; (iv) it is required that $\signt(a_1)=-1$ (equivalently one could also require that $\signt(a_1)=+1$ instead). We will call an oriented rational link diagram in this form the {\em alternative standard form} to distinguish it from the previously defined preferred standard form, and denote it by $\bb(q,p)$. Define $b_i=\signc(a_i) a_i$, and $(b_1,b_2,...,b_{2k+1})$ is called the {\em signed vector} of the rational link $\bb(q,p)$. Under these assumptions the authors proved the following theorem in~\cite{DL}:

\begin{theorem}\label{2bridge_theorem}
Let $K=\bb(q,p)$ be a rational link diagram in an alternative standard form with signed vector $(b_1,b_2,...,b_{2k+1})$, then the braid index of $K$ is given by
\begin{equation}\label{2bridgeformula}
\textbf{b}(K)=1+\frac{2+\signc(b_1)+\signc(b_{2k+1})}{4}+\sum_{b_{2j}>0,1\le j\le k}\frac{b_{2j}}{2}+\sum_{b_{2j+1}<0,0\le j\le k}\frac{|b_{2j+1}|}{2}
\end{equation}
\end{theorem}

\begin{remark}
If $\signt(a_1)=+1$ is used in the definition of the alternative standard form then the formula in Theorem~\ref{2bridge_theorem} becomes 
\begin{equation}\label{2bridgeformula_2}
\textbf{b}(K)=1+\frac{2-\signc(b_1)-\signc(b_{2k+1})}{4}+\sum_{b_{2j}<0,1\le j\le k}\frac{|b_{2j}|}{2}+\sum_{b_{2j+1}>0,0\le j\le k}\frac{b_{2j+1}}{2}.
\end{equation} 
\end{remark}

\comment{
\begin{remark}
\label{lastsign}{\em 
The conventions employed will only allow certain crossing signs for $a_{2k+1}$ under the assumption that all $a_i$ are positive. $pq$ is odd if and only if $\signc(a_{2k+1})=1$. This happens because if $pq$ is odd then we must have a knot and $p/q$ must be a parity $(1)$-tangle. Similarly $p$ is even and $q$ is odd if and only if  $\signc(a_{2k+1})=-1$. This happens because we must also have a knot and $p/q$ must be a parity $(0)$-tangle. An example of this second case is shown in Figure~\ref{parityinfinity}. If $q$ is even and $p$ is odd then we have a link ($p/q$ must be a parity $(\infty)$-tangle) and the crossings in $a_{2k+1}$ can have either sign because the second component can have either orientation. However because of the choice of the orientation we must have $\signc(a_{1})=-1$.}
\end{remark}

\begin{figure}[htb!]
\begin{center}
\includegraphics[scale=.5]{fig15}
\end{center}
\caption{Left: A standard drawing of the rational tangle $56/191 = A(3,2,2,3,3)$; Right: The denominator $D(A(3,2,2,3,3))$ is a standard two bridge link diagram of the two bridge link $b(191,58)=b(191,56)$.}
\label{parityinfinity}
\end{figure}
}

\begin{theorem}\label{main_theorem_even}
For an oriented rational link $K=\b(q,p)$ with $pq$ even, the braid index computed from a diagram in a preferred standard form $\tilde{b}(q,p)$ is equal to the braid index computed from a diagram in an alternative standard form $\bb(q,p)$, {\em i.e.}, the formulation of the braid index of $K$ given by Theorem~\ref{2bridge_theorem} is equivalent to the Cromwell-Murasugi index given by Theorem~\ref{thm:braid}.
\end{theorem}

\begin{proof}
Without loss of generality assume that $\signt(a_1)=-1$ is used in the definition of the alternative standard form and formula~(\ref{2bridgeformula}) is used. Since the diagram is also in a preferred standard form, it implies that $K= \bb(q,p)=\tb(q,-p)=(0;-a_1,\ldots,-a_n)$, hence $\signc(B_1)=-1=\signc(b_1)$. Thus the formula of Theorem~\ref{thm:braid} looks like
$$
\mathbf{b}(K)=1+\frac{1}{2}\left( \sum\limits_{j\geq 0, \signc(B_{2j+1})=-1} a_{2j+1}
  +\sum\limits_{j\geq 1, \signc(B_{2j})=1} a_{2j}\right)+c(p/q,n).
$$
%We note that the term $c(p/q,n)$ is not effected by this choice.

Case 1: $n=2k+1$.
If $\signc(B_1) = \signc(B_{2k+1})$ then $\signc(b_1)= \signc(b_{2k+1})=-1$ and $c(p/q,n)=0= \frac{2+\signc(b_1)+\signc(b_{2k+1})}{4}$.
If $\signc(B_1) \ne \signc(B_{2k+1})$ then $\signc(b_1)+ \signc(b_{2k+1})=0$ and $c(p/q,n)=1/2= \frac{2+\signc(b_1)+\signc(b_{2k+1})}{4}$.
Moreover
$$
\frac{1}{2}\left(\sum\limits_{j\geq 0, \signc(B_{2j+1})=-1} a_{2j+1}
  +\sum\limits_{j\geq 1, \signc(B_{2j})=1} a_{2j}\right)=\sum_{b_{2j+1}<0,0\le j\le k}\frac{|b_{2j+1}|}{2}+\sum_{b_{2j}>0,1\le j\le k}\frac{b_{2j}}{2},
$$
and we have equality between Theorems~\ref{thm:braid} and~\ref{2bridge_theorem}.

Case 2: $n=2k$. In this case it is necessary that $a_{2k}>1$ and $\bb(q,p)=(0,a_1,a_2,...,a_{2k}-1,1)$. Furthermore, $b_{2k}=\signc(a_{2k})(a_{2k}-1)$, $a_{2k+1}=1$ and $\signc(a_{2k+1})=\signc(a_{2k})$ as one can easily check. So formula~(\ref{2bridgeformula}) becomes
$$
1+\frac{1}{2}+\sum_{b_{2j}>0,1\le j\le k}\frac{a_{2j}}{2}+\sum_{b_{2j+1}<0,0\le j\le k-1}\frac{a_{2j+1}}{2}
$$
if $\signc(a_{2k})=-1$ and 
$$
1+\sum_{b_{2j}>0,1\le j\le k}\frac{a_{2j}}{2}+\sum_{b_{2j+1}<0,0\le j\le k-1}\frac{a_{2j+1}}{2}
$$
if $\signc(a_{2k})=+1$.
On the other hand, we have $\tb(q,-p)=(0,-a_1,-a_2,...,-a_{2k})$.
If $-1=\signc(B_{1})\ne\signc(B_{2k})$, then $c(p/q,2k)=0$ and~(\ref{mid_formula}) becomes
% however since $\signc(b_1)\ne \signc(b_{2k+1})=-1$ we have $\frac{2+\signc(b_1)+\signc(b_{2k+1})}{4}=1/2$. 
\begin{eqnarray*}
\mathbf{b}(K)&=&1+\frac{1}{2}\left( \sum\limits_{j\geq 0, \signc(B_{2j+1})=-1} a_{2j+1}
  +\sum\limits_{j\geq 1, \signc(B_{2j})=+1} a_{2j}\right)\\
&=&
1+\sum_{b_{2j}>0,1\le j\le k}\frac{a_{2j}}{2}+\sum_{b_{2j+1}<0,0\le j\le k-1}\frac{a_{2j+1}}{2}.
\end{eqnarray*}
If $-1=\signc(B_{1})=\signc(B_{2k})$ then $c(p/q,2k)=1/2$ and~(\ref{mid_formula}) becomes 
\begin{eqnarray*}
\mathbf{b}(K)&=&
1+\frac{1}{2}+\frac{1}{2}\left(\sum\limits_{j\geq 0, \signc(B_{2j+1})=-1} a_{2j+1}
  +\sum\limits_{j\geq 1, \signc(B_{2j})=+1} a_{2j}\right)\\
 &=& 1+\frac{1}{2}+\sum_{b_{2j}>0,1\le j\le k}\frac{a_{2j}}{2}+\sum_{b_{2j+1}<0,0\le j\le k-1}\frac{a_{2j+1}}{2}.
\end{eqnarray*}
Thus we have equality between Theorems~\ref{thm:braid} and~\ref{2bridge_theorem}. 
\end{proof}

Theorem~\ref{main_theorem_even} settles the case when $\b(q,p)$ has a preferred standard rational link diagram. What if $\b(q,p)$ does not have a preferred standard rational link diagram? Since Theorem~\ref{thm:braid} cannot be applied directly to it, we cannot compare the formula~(\ref{2bridgeformula}) obtained directly from  
$\bb(q,p)$ with formula~(\ref{mid_formula}). However, recall that if $\bb(q,p)$ is not in the preferred standard form, then its mirror image can be deformed into a preferred standard form $\tb(q,q-p)$ which is also in an alternative standard form (one may have to adjust $a_n$ to create a signed vector of odd length), by the procedure illustrated in Figure~\ref{fig:preferredform}. 
In the following we show that there is an elementary proof showing that the braid index formula~(\ref{2bridgeformula}) based on the diagram $\bb(q,p)$ is the same as the one based on the diagram $\bb(q,q-p)$. 

Let $(b_1, b_2, ..., b_{2k+1})$ be the signed vector of $b(q,p)$ in its alternate standard form. The signed vector $(c_1,c_2,...,c_{2k'+1})$ of $b(q,q-p)$ depends on the values of $b_1$ and $b_{2k+1}$ and is given in the following cases:  
$$
\begin{array}{ll}
-\big(\signc(b_1),(\signc(b_1)(|b_1|-1),b_2,...,b_{2k}, (\signc(b_{2k+1})(|b_{2k+1}|-1), \signc(b_{2k+1})\big),& |b_1|>1, |b_{2k+1}|>1;\\
-\big((\signc(b_2)(|b_2|+1),b_3,...,b_{2k}, (\signc(b_{2k+1})(|b_{2k+1}|-1), \signc(b_{2k+1})\big),& |b_1|=1, |b_{2k+1}|>1;\\
-\big((\signc(b_1),(\signc(b_1)(|b_1|-1),b_2,...,b_{2k-1}, (\signc(b_{2k})(|b_{2k}|+1)\big),& |b_1|>1, |b_{2k+1}|=1;\\ 
-\big((\signc(b_2)(|b_2|+1),b_3,...,b_{2k-1}, (\signc(b_{2k})(|b_{2k}|+1)\big),& |b_1|=1, |b_{2k+1}|=1.
\end{array}
$$
By~(\ref{2bridgeformula}), we have 
\begin{equation}\label{toverify}
\textbf{b}(b(q,p))=1+\frac{2+\signc(b_1)+\signc(b_{2k+1})}{4}+\sum_{b_{2j}>0,1\le j\le k}\frac{b_{2j}}{2}+\sum_{b_{2j+1}<0,0\le j\le k}\frac{|b_{2j+1}|}{2}.
\end{equation}
On the other hand, one can verify that applying~(\ref{2bridgeformula}) to the signed vector $(c_1,c_2,...,c_{2k'+1})$ for $b(q,q-p)$ as listed above yields exactly the braid index. We will verify the case when $|b_1|>1$ and $|b_{2k+1}|>1$, and leave the other cases to the reader. Notice that if we denote the signed vector of $b(q,q-p)$ by $(c_1,c_2,...,c_{2k+3})$, then we have $c_1=-\signc(b_1)$, $c_2=-\signc(b_1)(|b_1|-1)$, $c_m=-b_{m-1}$ for $3\le m\le 2k+1$, $c_{2k+2}=-\signc(b_{2k+1})(|b_{2k+1}|-1)$, $c_{2k+3}=-\signc(b_{2k+1})$. By~(\ref{2bridgeformula}) we have
\begin{eqnarray}
\textbf{b}(b(q,q-p))&=&1+\frac{2+\signc(c_1)+\signc(c_{2k+3})}{4}+\sum_{c_{2j}>0,1\le j\le k+1}\frac{c_{2j}}{2}+\sum_{c_{2j+1}<0,0\le j\le k+1}\frac{|c_{2j+1}|}{2}\nonumber\\
&=&
1+\frac{2-\signc(b_1)-\signc(b_{2k+1})}{4}+\sum_{b_{2j+1}<0,1\le j\le k-1}\frac{|b_{2j+1}|}{2}+\sum_{b_{2j}>0,2\le j\le k}\frac{b_{2j}}{2}\nonumber\\
&+&\Delta(c_1)+\Delta(c_2)+\Delta(c_{2k+2})+\Delta(c_{2k+3}),\label{contribution}
\end{eqnarray}
where $\Delta(c_i)$ is the contribution of the crossings corresponding to $c_i$ ($i=1,2,2k+2,2k+3$). There are four cases to verify: (i) $b_1>1$, $b_{2k+1}>1$; (ii) $b_1>1$, $b_{2k+1}<-1$; (iii) $b_1<-1$, $b_{2k+1}>1$; (iv) $b_1<-1$, $b_{2k+1}<-1$.

(i) $c_1=-1$, $c_2=-(b_1-1)<0$, $c_{2k+3}=-1$ and $c_{2k+2}=-(b_{2k+1}-1)<0$. It follows that $\Delta(c_1)=1/2=2\signc(b_1)/4$, $\Delta(c_2)=0$, $\Delta(c_{2k+2})=0$ and $\Delta(c_{2k+3})=1/2=2\signc(b_{2k+1})/4$, thus the summation in~(\ref{contribution}) equals the summation in~(\ref{toverify}).

(ii) $c_1=-1$, $c_2=-(b_1-1)<0$, $c_{2k+3}=1$ and $c_{2k+2}=(|b_{2k+1}|-1)>0$. It follows that $\Delta(c_1)=1/2=2\signc(b_1)/4$, $\Delta(c_2)=0$, $\Delta(c_{2k+2})=(|b_{2k+1}|-1)/2=|b_{2k+1}|/2-1/2=|b_{2k+1}|/2+2\signc(b_{2k+1})/4$ and $\Delta(c_{2k+3})=0$, thus the summation in~(\ref{contribution}) also equals the summation in~(\ref{toverify}).

(iii) and (iv) are similar and left to the reader.

This proves the following:

\begin{theorem}\label{main_theorem_odd}
Assume that $p/q=[0,a_1,\ldots,a_n]$ and $\b(q,p)$ does not have a preferred standard form. Then the formulation of the braid index of $\bb(q,p)$ given by Theorem~\ref{2bridge_theorem} is equivalent to the Cromwell-Murasugi index given in Theorem~\ref{thm:braid} applied to mirror image $\bb(q,q-p)$.
\end{theorem}

Thus we have established the equivalence of the braid index formulations given by Theorem~\ref{2bridge_theorem} and Theorem~\ref{thm:braid} respectively, through completely elementary arguments.

We conclude this section with presenting an extended example of a
rational link with two components to illustrate the various ways to
compute the braid index of a rational link, depending on how the link is
represented. 

\begin{figure}[htb!]
\includegraphics[scale=1]{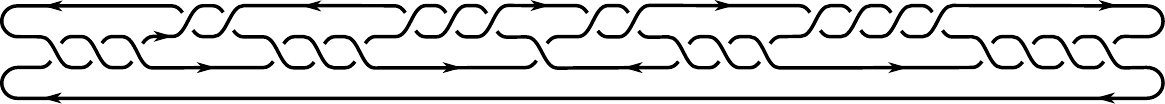}
\caption{The two bridge link $\bb(17426,5075)=(0,3,2,3,3,1,2,3,4,4)$ in an alternative standard form with signed vector $(3,2,3,3,-1,-2,-3,4,-4)$. Notice that it is not in a preferred standard form.}
\label{2bridgeone}
\end{figure}

\begin{example}{\em
Consider the rational link given in Figure~\ref{2bridgeone}, which is in an alternative standard form with signed vector $(3,2,3,3,-1,-2,-3,4,-4)$, but not in a preferred standard form. By Theorem~\ref{2bridge_theorem}, we have
\begin{eqnarray*}
\textbf{b}(K)&=&1+(2+\signc(b_1)+\signc(b_{2k+1}))/4+\sum_{b_{2j}>0,1\le j\le k}b_{2j}/2+\sum_{b_{2j+1}<0,0\le j\le k}|b_{2j+1}|/2\\
&=&1+1/2+(2+3+4)/2+(1+3+4)/2=10.
\end{eqnarray*}
Notice that we cannot apply Theorem~\ref{thm:braid} to $\bb(17426,5075)$ since it is not in a preferred standard form, but we can deform it to $\bb(17426,-12351)$ with vector $(-1,-2,-2,-3,-3,-1,-2,-3,-4,-3,-1)$ and $\signt(a_1)=+1$.
The diagram given by this is in preferred standard form $\tb(17426,12351)$ with $12351/17426=[0,1^+, 2^+, 2^+, 3^+, 3^+, 1^-, 2^-, 3^-, 4^+, 4^-]$. By Theorem~\ref{thm:braid} we have
\begin{eqnarray*}
\textbf{b}(K)&=&1+c(p/q,n)+\frac{1}{2}\left( \sum\limits_{j\geq 0, \signc(B_{2j+1})=\signc(B_1)} a_{2j+1}
  +\sum\limits_{j\geq 1, \signc(B_{2j})=-\signc(B_1)} a_{2j}\right)\\
&=&1+0+(1+2+3+4)/2+(1+3+4)/2=10.
\end{eqnarray*}
The mirror image $\overline{K}$ of $\bb(17426,5075)$ is also in an alternative standard form with signed vector $(-1,-2,-2,-3,-3,1,2,3,-4,3,1)$. By Theorem~\ref{2bridge_theorem} we have:
\begin{eqnarray*}
\textbf{b}(\overline{K})&=&1+(2+\signc(b_1)+\signc(b_{2k+1}))/4+\sum_{b_{2j}>0,1\le j\le k}b_{2j}/2+\sum_{b_{2j+1}<0,0\le j\le k}|b_{2j+1}|/2\\
&=&1+1/2+(1+3+3)/2+(1+2+3+4)/2=10.
\end{eqnarray*}
We also have $12351/17426=[0,-2, 2, -2, 6, -2, 2, -4, 4, 4, -2, 4]$ and the Cromwell-Murasugi index as in Definition \ref{def:CM} is given by $16-7+1=10$.
Finally, $12351/17426=[0,1, 2, 2, 3, 3, 1, 2, 3, 4, 4]$  has a primitive block decomposition with no nontrivial primitive block as follows:
 $[0,\mathbf{1},2,\mathbf{2},3,\mathbf{3};\mathbf{1},2, \mathbf{3};\mathbf{4};\mathbf{4}]$
By Theorem~\ref{thm:CMpblock}, the Cromwell-Murasugi index
of $12351/17426$ is  
$
\mathbf{b}(K)=(1+2+3+1+3+4+4)/2+1=10.
$

If we reverse the orientation of one of the components in 
$\bb(17426,5075)$ given in Figure~\ref{2bridgeone}, then we obtain a link diagram that is in a preferred standard form $K^\p=\tb(17426,5075)$ which is also in an alternative standard form, but with a different signed vector $(-3, -2, -3, 3, 1, 2, 3, 4, 4)$. The braid index of $K^\p$ can be computed both by the formula in Theorem~\ref{2bridge_theorem} and by the formula in Theorem~\ref{thm:braid}:
\begin{eqnarray*}
\textbf{b}(K^\p)&=&1+(2+\signc(b_1)+\signc(b_{2k+1}))/4+\sum_{b_{2j}>0,1\le j\le k}b_{2j}/2+\sum_{b_{2j+1}<0,0\le j\le k}|b_{2j+1}|/2\\
&=&1+1/2+(3+2+4)/2+(3+3)/2=9\\
&=&1+c(p/q,n)+\sum\limits_{j\geq 0, \signc(B_{2j+1})=\signc(B_1)} a_{2j+1}/2
  +\sum\limits_{j\geq 1, \signc(B_{2j})=-\signc(B_1)} a_{2j}/2.
\end{eqnarray*}
On the other hand, we have $5075/17426=[0,4, -2, 4, 4, -4, 2, -2, 6, -2, 2, -2]$ and the Cromwell-Murasugi index as in Definition~\ref{def:CM} gives $17-9+1=9$.
Notice that although $5075/17426 =[0,3, 2, 3, 3, 1, 2, 3, 4, 4]$  has no primitive block decomposition, $5075/17426=[0,3, 2, 3, 3, 1, 2, 3, 4, 3,1]$  does have a primitive block decomposition with no nontrivial primitive block: 
 $[0,\mathbf{3},2,\mathbf{3}; \mathbf{3},1,\mathbf{2},3,\mathbf{4},3,$ $\mathbf{1}].$
By Theorem~\ref{thm:CMpblock}, the Cromwell-Murasugi index
of $5075/17426$ is  
$
\mathbf{b}(K)=(3+3+3+2+4+1)/2+1=9.
$
}
\end{example}

\section{The HOMFLY Polynomial of alternating rational links}
\label{HOMFLY_section}

In this section we derive a formula for the HOMFLY
polynomial of an alternating link.
We will rely on a result of Lickorish and Millett~\cite[Proposition
  14]{Li-Mi} about the HOMFLY polynomial of a rational link whose
continued fraction is represented in the all-even form. Note that they
use the parameters $l$ and $\mu$ which are 
linked to the parameters $a$ and $z$ we use below by the formulas
$$
l=\mathbf{i}\cdot a \quad\mbox{and}\quad \mu=\frac{a^2-1}{az}.
$$

Using the matrices
\begin{equation}
\label{eq:M2r}  
\M(2r)=
\begin{pmatrix}
(1-(-1)^r l^{-2r})\mu^{-1} & (-1)^r l^{-2r}\\
1 & 0\\
\end{pmatrix}
=
\begin{pmatrix}
\frac{(1-a^{-2r})az}{a^2-1}  & a^{-2r}\\
1 & 0\\
\end{pmatrix}
\end{equation}
for all even integers, they state the following theorem.  

\begin{proposition}[Lickorish--Millett]
\label{prop:LM}
Let $K$ be a rational knot or link, represented by the continued
fraction $[0,c_1,\ldots,c_n]$ where the $c_i$ are even integers. Then
the HOMFLY polynomial ${\mathcal P}(K)$ is given by
\begin{equation}
\label{eq:LM}
{\mathcal P}(K)=
\begin{pmatrix}
1 & 0
\end{pmatrix}  
\M((-1)^n c_n)\M((-1)^{n-1} c_{n-1})\cdots \M(c_2)\M(-c_1)
\begin{pmatrix}
  1\\
  \frac{a^2-1}{az}\\
\end{pmatrix}
\end{equation}
\end{proposition}  
\begin{remark}
{\em Lickorish and Millett~\cite{Li-Mi} represent rational knots and
  links by continued fractions of the form $[c_n,\ldots,c_1]$ instead of
  our $[0,c_1,\ldots,c_n]$. To account for this change we reversed the
  order of indices in their formula. They also introduce the conjugate
  $\overline{M(2r)}$ of $M(2r)$ obtained by interchanging
  $l=\mathbf{i}\cdot z$ and 
  $l^{-1}=-\mathbf{i} z^{-1}$. Note however, that interchanging  $l$ and
  $l^{-1}$ in~\eqref{eq:M2r} 
  takes $\M(2r)$ into $\M(-2r)$: the parity of $r$ is the same as that
  of $-r$, thus $(-1)^r=(-1)^{-r}$ whereas $l^{-2(-r)}=l^{2r}$, so
  replacing $2r$ with $-2r$ has the same effect as interchanging $l$ and
  $l^{-1}$.  Hence taking the conjugate of $\M(c_i)$
is exactly the same as replacing it with $\M(-c_i)$. The formula stated
in~\cite[Proposition 14]{Li-Mi} calls for taking the conjugate of every
second matrix factor in such a way that the rightmost matrix is
conjugated. The effect of replacing each $\M(c_i)$ with $\M((-1)^ic_i)$
is exactly the same.
}
\end{remark}  

Keeping Proposition~\ref{prop:pblock} in mind, we will be interested in
using Proposition~\ref{prop:LM} in situations where a contiguous
substring $c_i, 
c_{i+1}, \ldots,c_j $ of partial denominators is of the form
$2,-2,\ldots,(-1)^{j-i} 2$
or $-2,2,\ldots, ,(-1)^{j-i+1} 2$. Note that, due to the rule calling
for ``conjugating'' every second factor in~\eqref{eq:LM}, such
alternating strings of $2$s and $-2$s give rise to powers of the
matrices $\M(2)$ and $\M(-2)$ respectively. Substituting $r=1$ and
$r=-1$, respectively, into \eqref{eq:M2r} yields

\begin{equation}
\label{eq:M2}  
\M(2)=\begin{pmatrix}
za^{-1} & a^{-2}\\
1 & 0\\
\end{pmatrix}
\quad\mbox{and}\quad
\M(-2)=\begin{pmatrix}
-za & a^{2}\\
1 & 0\\
\end{pmatrix}.
\end{equation}

It is easy to describe the powers of such matrices in terms of {\em
  Fibonacci polynomials $F_n(x)$} defined by the initial conditions
\begin{equation}
F_0(x)=0\quad\mbox{and}\quad F_1(x)=1;
\end{equation}  
and the recurrence
\begin{equation}
\label{eq:fibrec}
  F_{n+1}(x)=x F_n(x)+F_{n-1}(x) \quad\mbox{for $n\geq 1$} 
\end{equation}  

It is well known that these polynomials are given by the closed form
formula
\begin{equation}
F_n(x)=\sum_{j=0}^{\lfloor (n-1)/2\rfloor} \binom{n-j-1}{j} x^{n-1-2j}.
\end{equation}

\begin{proposition}
\label{prop:matrixp}
  For all $n\geq 1$, we have
  $$
\begin{pmatrix}
u & v^2\\
1 & 0\\
\end{pmatrix}^n
=
\begin{pmatrix}
v^n \cdot F_{n+1}\left(u/v\right)& v^{n+1}\cdot F_{n}\left(u/v\right)\\
v^{n-1}\cdot F_{n}\left(u/v\right)& v^{n}\cdot F_{n-1}\left(u/v\right)\\
\end{pmatrix}
$$
\end{proposition}  
\begin{proof}
We proceed by induction on $n$. For $n=1$ the statement is is a direct
consequence of the definitions. Assume the statement is true for some
$n\geq 1$. Then
\begin{align*}
\begin{pmatrix}
u & v^2\\
1 & 0\\
\end{pmatrix}^{n+1}
&=\begin{pmatrix}
u & v^2\\
1 & 0\\
\end{pmatrix}
\cdot
\begin{pmatrix}
u & v^2\\
1 & 0\\
\end{pmatrix}^{n}
=
\begin{pmatrix}
u & v^2\\
1 & 0\\
\end{pmatrix}
\cdot
\begin{pmatrix}
v^n \cdot F_{n+1}\left(u/v\right)& v^{n+1}\cdot F_{n}\left(u/v\right)\\
v^{n-1}\cdot F_{n}\left(u/v\right)& v^{n}\cdot F_{n-1}\left(u/v\right)\\
\end{pmatrix}\\
&=
\begin{pmatrix}
v^{n+1} (u/v\cdot F_{n+1}(u/v)+F_n(u/v)) & v^{n+2} (u/v\cdot
F_{n}(u/v)+F_{n-1}(u/v))\\
v^n \cdot F_{n+1}\left(u/v\right) & v^{n+1}\cdot F_{n}\left(u/v\right)\\
\end{pmatrix}.
\end{align*}
and the statement is now a direct consequence of the recurrence~(\ref{eq:fibrec}). 
\end{proof}  

Direct substitution of Proposition~\ref{prop:matrixp} into~(\ref{eq:M2})
yields

\begin{align}
\label{eq:M2p}  
\M(2)^n&=
\begin{pmatrix}
a^{-n} \cdot F_{n+1}\left(z\right)& a^{-(n+1)}\cdot F_{n}\left(z\right)\\
a^{-n+1}\cdot F_{n}\left(z\right)& a^{-n}\cdot F_{n-1}\left(z\right)\\
\end{pmatrix}\\
\M(-2)^n&=
\label{eq:M-2p}  
\begin{pmatrix}
a^n \cdot F_{n+1}\left(-z\right)& a^{n+1}\cdot F_{n}\left(-z\right)\\
a^{n-1}\cdot F_{n}\left(-z\right)& a^{n}\cdot F_{n-1}\left(-z\right)\\
\end{pmatrix}.
\end{align}

We extend the validity of \eqref{eq:M2p} and \eqref{eq:M-2p} to $n=0$ by
setting
\begin{equation}
\label{eq:Fneg}
F_{-1}(x)=1.
\end{equation}
Observe that \eqref{eq:Fneg} is consistent with applying the recurrence formula 
\eqref{eq:fibrec} to $n=0$, as we have $1=F_1(x)=x\cdot
F_{0}(x)+F_{-1}(x)$. Furthermore, extending \eqref{eq:M2p} and
\eqref{eq:M-2p} in such a way yields  
$$
\M(2)^0=\M(-2)^0=
\begin{pmatrix}
1& 0\\
0& 1\\
\end{pmatrix}
$$
as expected. 

We can use this to derive a different formula for the HOMFLY polynomial
${\mathcal P}(K)$ than the one given in Proposition~\ref{prop:LM}.
Let $K$ be a rational knot or link, represented by the continued
fraction $p/q =[0,c_1,\ldots,c_n]$ where $pq$ is an even integer and the vector $(c_1,\ldots,c_n)$ describes a standard diagram with all $c_i>0$. Let $(Bl_1;\ldots;Bl_n)$ be block decomposition of $(c_1,\ldots,c_n)$ and let $\sigma_i=\pm1$ be the sign of crossings in the block $Bl_i$. If $\M(Bl_i)$ is the $2\times 2$ matrix giving the contribution of the block $Bl_i$ to the HOMFLY polynomial ${\mathcal P}(K)$ then
the HOMFLY polynomial ${\mathcal P}(K)$ is given by
\begin{equation}
\label{eq:newmatrix}
{\mathcal P}(K)=
\begin{pmatrix}
1 & 0
\end{pmatrix}  
\M(Bl_1)\M(Bl_2)\cdots \M(Bl_k)
\begin{pmatrix}
  1\\
  \frac{a^2-1}{az}\\
\end{pmatrix},
\end{equation}
where for a block $Bl_i=(c_m,\ldots,c_{m+2j_i})$ we have the following product of $2j_i+1$ of $2\times 2$ matrices

\begin{eqnarray}
\label{eq:newblock}
\begin{aligned}
\M(Bl_i)=
\M\left(\sigma_i \frac{c_m+1}{2}\right) \M(\sigma_i 2)^{c_{m+1}-1}\M\left(\sigma_i \frac{c_{m+2}+2}{2}\right)  \M(\sigma_i 2)^{c_{m+3}-1} \\ 
\cdots \M(\sigma_i 2)^{c_{m+2j_i-1}-1}\M\left(\sigma_i \frac{c_{m+2j_i}+1}{2}\right).
\end{aligned}
\end{eqnarray}

The formula given in equation~(\ref{eq:newmatrix}) allows the computation of
the HOMFLY polynomial of an oriented link from its minimal alternating
form, provided it is in preferred standard form. 

\begin{example}{\em
We are given the two bridge knot $3244/4195$ with the primitive block decomposition $(0,1,3,2,2,3;5,3,3)$.
If we assume that $\signt(a_1)=-1$ then the crossings in the first block  $Bl_1=(1,3,2,2,3)$ are negative and the crossings in the second block $Bl_2= (5,3,3)$ are positive. Then we compute the following matrix product for the HOMFLY polynomial

\begin{eqnarray*}
{\mathcal P}(K)=
\begin{pmatrix}
1 & 0
\end{pmatrix} 
\M\left(- \frac{2}{2}\right) \M(- 2)^{3-1}\M\left(- \frac{4}{2}\right)
\M(- 2)^{2-1}\\ \M\left(- \frac{4}{2}\right) \M\left(\frac{6}{2}\right) \M(
2)^{3-1}\M\left(\frac{4}{2}\right)  
\begin{pmatrix}
  1\\
  \frac{a^2-1}{az}\\
\end{pmatrix}.
\end{eqnarray*}

%We obtain --I am not sure we need to show this
%\begin{eqnarray*}
%&&{\mathcal P}(K)=
%-\frac{\left(z^2+1\right) \left(z^2+2\right)^2 z^2}{a^8}-a^6 \left(z^8+6 z^6+12 z^4+9 z^2+2\right)+\frac{z^{10}+3 z^8-2 z^6-12 z^4-9 z^2-2}{a^6}+\\
%&&a^4 \left(z^{10}+4 z^8+2 z^6-6 z^4-4 z^2+1\right)+\frac{4 z^{10}+20 z^8+32 z^6+17 z^4+2 z^2+1}{a^4}+\\
%&&a^2 \left(4 z^{10}+22 z^8+41 z^6+31 z^4+9 z^2+1\right)+\frac{7 z^{10}+38 z^8+70 z^6+51 z^4+13 z^2+1}{a^2}+\\
%&&7 z^{10}+39 z^8+74 z^6+56 z^4+15 z^2+1
%\end{eqnarray*}
}
\end{example}

Using Theorem~\ref{thm:blocktrans} we obtain the following result. 
\begin{theorem}
\label{thm:homfly}  
Suppose $\tilde{b}(q,p)$ is
represented by a nonalternating 
continued fraction $p/q=[0,a_1,\ldots,$ $a_n]$
that has a primitive block decomposition with no exceptional primitive
block. Then the HOMFLY polynomial may be written in matrix form
as follows:
 $$
{\mathcal P}(K)=
\begin{pmatrix}
1 & 0
\end{pmatrix}  
\H\left(a_n\right) \H\left(a_{n-1}\right)\cdots \H\left(a_1\right)
\begin{pmatrix}
  1\\
  \frac{a^2-1}{az}\\
\end{pmatrix}.
$$
Here, after introducing $s=\sign(a_1)$, the matrices $\H\left(a_1\right),\H\left(a_{2}\right),\ldots
\H\left(a_n\right)$ are given by the following formulas.  
\begin{enumerate}
\item 
  $$\H(a_1)
  =
  \begin{cases}
\M(-a_1)& \mbox{if $a_1$ is even;}\\
\M(-(a_1+s))& \mbox{if $a_1$ is odd.}\\
  \end{cases}  
  $$
\item If $\signc(a_i)\neq \signc(a_{i-1})$ then set
  $$\H(a_i)
  =
  \begin{cases}
\M(-\signc(a_i) a_i)& \mbox{if $a_i$ is even;}\\
\M(-\signc(a_i) (a_i+s))& \mbox{if $a_i$ is odd.}\\
  \end{cases}  
  $$
\item If $\signc(a_i)=\signc(a_{i-1})=(-1)^{i-1}\cdot s$
then set
  $$\H(a_i)
  =
  \begin{cases}
\M(-\signc(a_i) (a_i+2s))& \mbox{if $a_i$ is even;}\\
\M(-\signc(a_i) (a_i+s))& \mbox{if $a_i$ is odd.}\\
  \end{cases}  
  $$
\item If $\signc(a_i)=\signc(a_{i-1})=(-1)^{i}\cdot s$ then  
  set
  $$
  \H(a_i)=
  \begin{pmatrix}
a^{\signc(a_i) \cdot (a_i-s)} \cdot F_{|a_i|+1}\left(-\signc(a_i)\cdot z\right)&
a^{\signc(a_i) \cdot a_i}\cdot F_{|a_i|}\left(-\signc(a_i)\cdot z\right)\\
a^{\signc(a_i) \cdot (a_i-2s)}\cdot F_{|a_i|}\left(-\signc(a_i)\cdot z\right)& a^{\signc(a_i) \cdot (a_i-s)}\cdot F_{|a_i|-1}\left(-\signc(a_i)\cdot z\right)\\
\end{pmatrix}.
  $$
\end{enumerate}
\end{theorem}  
The application of Theorem~\ref{thm:blocktrans} may be facilitated by
the following observation: if we start with a nonalternating continued
fraction $[0,a_1,\ldots,a_n]$ in which all signs alternate, the signs
alternate in the corresponding continued fraction $[0,c_1,\ldots,c_m]$
in all-even form, {\em except} at the beginning of a new primitive
block, where the crossing sign changes. Since the Lickorish-Millet
formula \eqref{eq:LM} calls for associating $\M((-1)^i c_i)$ to $c_i$,
the actual signs of the parameters $2r$ inside $\M(2r)$ that we will be
using will be constant within each primitive block and the opposites of the
crossing signs. The details of the verification are left to the reader. 

\begin{remark}
{\em Theorem~\ref{thm:homfly} is directly applicable only when the link has a
diagram $\bb(q,p)$ in preferred standard form. If this is not the case,
and the link diagram does not have a preferred standard form, we may
apply the result to the mirror image $\bb(q,q-p)$. It is well-known
(see, for example, \cite[Theorem 10.2.3]{Cromwell}) that the HOMFLY
polynomial of an oriented link may be obtained by substituting $a^{-1}$
into $a$ in the HOMFLY polynomial of its mirror image.}   
 \end{remark}

\comment{
\begin{figure}[h]
%70%  
\begin{center}
\input{automaton2.pspdftex}
\end{center}
\caption{Automaton, helping to identify knots}
\label{fig:automaton2}
\end{figure}}

\section*{Acknowledgments}
This work was partially supported by a grant from the Simons Foundation
(\#514648 to G\'abor Hetyei).


\begin{thebibliography}{99}
%\bibitem{A} C.~Adams,
%{\em The Knot Book}, American Mathematical Soc., 1994.

\bibitem{BZ}
G.~ Burde, H.~ Zieschang and M.~ Heusener
{\em Knots}, De Gruyter Studies in Mathematics \textbf{5},  2013.

\bibitem{Co}
J.H.~ Conway, 
An enumeration of knots and links, and some of their algebraic
properties, 1970 Computational Problems in Abstract Algebra
(Proc. Conf., Oxford, 1967) pp. 329--358 Pergamon, Oxford 

\bibitem{Cromwell}
P.~ Cromwell,
Knots and links,
Cambridge University Press, 2004.

\bibitem{DL}
  Y.~ Diao, C.~ Ernst, G.~ Hetyei and P.~ Liu,
A diagrammatic approach for determining the braid index of alternating
links, preprint 2018. 


\bibitem{Du}
S.~ Duzhin and M.~Shkolnikov,
A formula for the HOMFLY polynomial of rational links,
{\it Arnold Math.\ J.\ } \textbf{1} (2015), 345--359.   

  
\bibitem{Eu}
L.~ Euler, 
Elements of algebra,
Translated from the German by John Hewlett, Reprint of the 1840
edition, with an introduction by C.\ Truesdell, Springer-Verlag, New
York, 1984, lx+593 pp. 

\bibitem{FW} J.~Franks and R.~Williams
{\em Braids and The Jones Polynomial}, Trans. Amer. Math. Soc., \textbf{303} (1987), 97--108.
  
 \bibitem{Fr} P.~Freyd, D.~Yetter, J.~Hoste, W.~Lickorish, K.~Millett and A. ~Ocneanu
{\em A New Polynomial Invariant of Knots and Links}, Bull. Amer. Math. Soc. (N.S.) \textbf{12} (1985), 239--246.
  
\bibitem{Ka-Go}
J.R.~  Goldman, and L.H.~ Kauffman, 
Rational tangles, 
{\it Adv.\ in Appl.\ Math.\ } {\textbf 18} (1997), 300--332. 

\bibitem{Ka-La}
L.H.~ Kauffman, and S.~ Lambropoulou, 
On the classification of rational tangles,
{\it Adv.\ in Appl.\ Math.} {\textbf 33} (2004), 199--237.   

\bibitem{Li-Mi}
W.B.R.~ Lickorish and Kenneth C.~ Millett, 
A polynomial invariant of oriented links.
{\it Topology} {\bf 26} (1987), 107--141.   


\bibitem{Mo} H.~Morton
{\em Seifert Circles and Knot Polynomials}, Math. Proc. Cambridge Philos. Soc. \textbf{99} (1986), 107--109.


\bibitem{Mu} K.~Murasugi
  On The Braid Index of Alternating Links,
  {\it Trans.\ Amer.\ Math.\ Soc.} \textbf{326} (1991), 237--260.
  
\bibitem{Pr} J.~Przytycki and P.~Traczyk {\em Conway Algebras and Skein Equivalence of Links}, Proc. Amer. Math. Soc. \textbf{100} (1987), 744--748. 
  
 \bibitem{Schubert}
 K.~ Schubert, Knoten mit zwei Br{\"u}cken, Mathematische Zeitschrift, \textbf{65} (1956), 133--170.
 
\bibitem{Ya} S.~Yamada
{\em The Minimal Number of Seifert Circles Equals The Braid Index of A Link}, Invent. Math. \textbf{89} (1987), 347--356.
  

\end{thebibliography}
\end{document}